\documentclass[11pt,reqno]{amsart}
\usepackage[a4paper, total={160mm, 257mm}, left=20mm, right=20mm, top=20mm, bottom=20mm]{geometry}
\usepackage{graphicx}
\usepackage{amsmath,amsfonts,amssymb,amsthm,amscd}
\usepackage{cite}
\usepackage{latexsym}
\usepackage{mathrsfs}
\usepackage{braket}
\usepackage{url}
\usepackage{xcolor}

\usepackage{xparse}

\makeatletter
\@namedef{subjclassname@2020}{%
  \textup{2020} Mathematics Subject Classification}
\makeatother

\usepackage[backref=page]{hyperref}

\numberwithin{equation}{section}

\newtheorem{thm}{Theorem}[section]

\newtheorem{lem}[thm]{Lemma}

\theoremstyle{definition}
\newtheorem{defn}{Definition}[section]
\theoremstyle{remark}
\newtheorem{rem}[thm]{Remark}

\renewenvironment{cases}{\left\{\begin{aligned}}{\end{aligned}\right.}	
\def\pt{\partial}

\newcommand{\ud}{\mathrm{d}}           
\newcommand{\ue}{\mathrm{e}}           
\newcommand{\R}{\mathbb{R}}            

\newcommand{\case}[1]{\left\{
	\begin{aligned}
		#1
	\end{aligned}\right.}

\newcommand{\sibb}[1]{\left(#1\right)}
\renewcommand{\set}[1]{\left\{#1\right\}}
\newcommand{\norm}[1]{\left| #1 \right|}

\newcommand{\map}[3]{#1:#2\to #3}

\NewDocumentCommand{\dderiv}{O{} m m}{\frac{\ud^{#1} #2}{\ud #3^{#1}}}     

\newcommand{\minus}{\backslash}
\newcommand{\intersect}{\cap}

\newcommand{\Union}{\bigcup}

\newcommand{\laplace}{\Delta}

\newcommand{\grad}{\nabla}

\newcommand{\qint}[3]{\int_{#1}#2\ud #3}
\renewcommand{\iint}[4]{\int_{#1}^{#2}#3\ud #4}

\renewcommand{\div}{\mathrm{div}}

\newcommand{\ball}[1][ ]{B_{#1}}

\newcommand{\finv}{\Lambda_{F}}
\newcommand{\Fgrad}{\grad^{F}}
\let\Fgrad\grad
\newcommand{\mdiv}{\operatorname{div}_m}
\newcommand{\legen}{\mathcal{L}}

\begin{document}
	
	\title[Liouville problems for $(p,q)$-Laplacian inequalities]{On Liouville problems for $(p,q)$-Laplacian inequalities on Finsler measure spaces}
	\author[T.~Wang]{Tiancheng~Wang}
	\address{School of Mathematics, Sichuan University, Chengdu 610065, Sichuan,  P.~R.~China}
	\email{\href{mailto:wtc303515022@gmail.com}{wtc303515022@gmail.com}}
	\author[C.~Xiong]{Changwei~Xiong}
	\address{School of Mathematics, Sichuan University, Chengdu 610065, Sichuan,  P.~R.~China}
	\email{\href{mailto:changwei.xiong@scu.edu.cn}{changwei.xiong@scu.edu.cn}}
	\author[D.~Zhao]{Dongjie~Zhao}
	\address{School of Mathematics, Sichuan University, Chengdu 610065, Sichuan,  P.~R.~China}
	\email{\href{mailto:kridomss@gmail.com}{kridomss@gmail.com}}
	\date{\today}
	\thanks{This research was supported by National Key R and D Program of China 2021YFA1001800 and NSFC (Grant no. 12171334).}
    \subjclass[2020]{{35R45}, {35J92}, {35K92}, {58J60}}
	\keywords{Liouville theorem, $(p,q)$-Laplacian, Finsler measure space}
	
	\begin{abstract}
		We study the Liouville property for nonnegative weak solutions of the $(p,q)$-Laplacian elliptic differential inequality
		\[ \laplace^{m}_{p}u(x)+\laplace^{m}_{q}u(x)+V(x)u^{s}(x)\leq 0 \]
		and the associated parabolic differential inequality
		\[ \pt_t u(x,t) \ge \laplace^{m}_{p}u(x,t)+\laplace^{m}_{q}u(x,t)+V(x,t)u^{s}(x,t) \]
		on a forward geodesically complete noncompact Finsler measure space $(M,F,m)$ with finite reversibility. Under several sets of integral growth conditions on the positive potential function over certain annular domains, we prove that any nonnegative weak solution vanishes almost everywhere. The proofs of our results are essentially based on the nonlinear capacity method.
	\end{abstract}
	\maketitle
	
\section{Introduction}
	The Liouville theorem for solutions of linear and nonlinear partial differential equations and inequalities has been a central topic since the first result given by Liouville and Cauchy (see \cite{SeZo02} for historical remarks). It has found significant applications in many theoretical aspects of partial differential equations, ranging from the regularity theory to quantitative properties; see, e.g., Section~1 in \cite{CiGo23} and references therein for an introduction. Let us first review some known Liouville-type results on elliptic and parabolic problems in the existing literature, and then present our main results of the paper.
\subsection{The elliptic case}  One systematic investigation on the Liouville property of entire solutions for nonlinear elliptic equations appeared in \cite{Serr72} by Serrin, in addition to the early contributions in \cite{BoJa62}. For the semilinear elliptic equation
	\[ \laplace u(x)+u^{s}(x)=0, \quad u(x)\geq 0, \quad x\in \mathbb{R}^{n}, \]
the celebrated work by Gidas and Spruck \cite{GiSp81} asserts that the solution when $n\geq 3$ is identically zero if $1\leq   s<(n+2)/(n-2)$; while it is known that for $s\geq (n+2)/(n-2)$ there exists a nontrivial nonnegative solution. They also treated the problem on complete Riemannian manifolds with nonnegative Ricci curvature. (See, e.g., \cite{Bida89} by Bidaut-V\'{e}ron for further developments in this direction.) In contrast, for the inequality
\begin{align*}
\laplace u(x)+u^{s}(x)\leq 0, \quad u(x)\geq 0, \quad x\in \mathbb{R}^{n},
\end{align*}
the solution when $n\geq 3$ is identically zero if $0<s\leq n/(n-2)$ (see, e.g., \cite{MiPo98}).
\begin{rem}
It is interesting to observe (see, e.g., p.~133 in \cite{MiPo04}) that in general, critical exponents (like $(n+2)/(n-2)$ and $n/(n-2)$ above) for elliptic equations and inequalities are different; while for parabolic equations and inequalities discussed below in Section~\ref{sec1.2} critical exponents coincide.
\end{rem}
Next, moving to more general framework, in a series of papers \cite{CaDM08,CaMP09,MiPo98,MiPo99,MiPo01,MiPo04,Pokh97,Poho09}, the authors of these papers studied the Liouville problems of elliptic inequalities on $\R^n$, by use of the nonlinear capacity method introduced by Mitidieri and Pohozaev \cite{MiPo98,MiPo01}.
\begin{rem}
The nonlinear capacity of differential operators was introduced by Pokhozhaev in \cite{Pokh97} and see \cite{MiPo04,Poho09} for an overview of the method using it. Roughly speaking, the method consists of first deriving a priori integral estimates for local weak solutions and then analyzing the asymptotic behavior of these estimates with respect to the related parameters of the problem (\cite{MiPo04,Poho09}).
\end{rem}
The nonlinear capacity method cannot be directly extended to Riemannian manifolds without imposing curvature conditions. To avoid assuming curvature conditions, motivated by Kurta's work \cite{Kurt99}, Grigor'yan and Kondratiev \cite{GrKo10} considered a variant of the method. This variant allows them to study the Liouville property of the solution to the divergence-form semilinear inequalities
\begin{align}\label{eq-GK}
\div\sibb{A(x)\grad u}+V(x)u^{s}\leq 0
\end{align}
with a nonnegative definite $(1,1)$-tensor $A(x)$ and a potential $V(x)$ on Riemannian manifolds, under only volume growth conditions.
\begin{rem}
The idea of using volume growth conditions to prove the Liouville property may date back to Cheng and Yau's work \cite{ChYa75}. There they proved that, in a complete Riemannian manifold $M$, if the volume of a geodesic ball grows at most like a quadratic polynomial, then positive superharmonic functions defined on $M$ must be constant functions. We also note that generally curvature conditions imply volume conditions (e.g., by Bishop--Gromov volume comparison theorem). Thus volume conditions are weaker than curvature conditions in some sense.
\end{rem}
\begin{rem}
There are so many works in the literature devoted to Liouville problems that we can hardly mention all of them. In this paper, we mainly focus on those which use or are related to (weighted) volume growth conditions.
\end{rem}
Furthermore, Grigor'yan and Sun \cite{GrSu14} succeeded in obtaining the Liouville-type result for \eqref{eq-GK} with the sharp volume growth condition when $A(x)$ is the identity map and $V(x)\equiv 1$. Some generalizations of the result \cite{GrSu14} are given by Sun \cite{Sun14,Sun15} and further developments involving Green's functions can be found in \cite{GrSu19}.
	
	
The above-mentioned nonlinear capacity method due to Mitidieri and Pohozaev \cite{MiPo98,MiPo01} can also allow them to study the quasilinear inequality
	\[  \laplace_p u + u^{s}\leq 0 \]
	on $\R^{n}$ and its exterior domains, where $\laplace_{p} u=\div\sibb{\norm{\grad u}^{p-2}\grad u}$ ($p>1$) is the  $p$-Laplacian operator. They established Liouville-type theorems under the condition $0<s\leq n(p-1)/(n-p)$ when $p<n$, as well as for inequalities with gradient-dependent nonlinearities. Subsequently, Bidaut-V\'{e}ron and Pohozaev \cite{BiPo01} as well as Serrin and Zou \cite{SeZo02} were able to obtain refined deep results in more general situations. The approach in \cite{SeZo02} is quite different, relying on an extensive employment of the comparison method and an estimate from below for the solution far from the origin. Back to the Riemannian setting, along the line indicated by Grigor'yan and Sun's \cite{GrSu14}, Mastrolia, Monticelli and Punzo \cite{MaMP15} studied the Liouville property for the more general elliptic differential inequality
\begin{align*}
\frac{1}{a(x)}\mathrm{div}(a(x)|\nabla u|^{p-2}\nabla u) +V(x) u^{s}\leq 0
\end{align*}
on a Riemannian manifold $M$, where $a(x)\in \mathrm{Lip}_{\mathrm{loc}}(M)$ is a positive function and $V(x)\in L^1_{\mathrm{loc}}(M)$ is a positive potential function. Some other generalizations of \cite{GrSu14, Sun15} were given in \cite{Zhao25b}.
	
	On the other hand, the study on the $(p,q)$-Laplacian $\laplace_{p}+\laplace_{q}$ has attracted considerable interest because of its importance in nonlinear elasticity and calculus of variations \cite{Zhik86,Zhik95,Marc89,Marc91}. The $(p,q)$-Laplacian is a differential operator with unbalanced growth and it provides models of strongly anisotropic materials \cite{Marc89,Zhik95}. Problems concerning the $(p,q)$-Laplacian come from reaction--diffusion systems of the form
\begin{align*}
\pt_t u=\mathrm{div}(B(u)\nabla u)+c(x,u),
\end{align*}
where the function $u$ represents the density or concentration of multicomponent substances, the differential term with $B(u)=|\nabla u|^{p-2}+|\nabla u|^{q-2}$ stands for the diffusion, and the term $c(x,u)$ describes the reaction relating to sources or loss processes. Meanwhile, it is worth noting that the Born--Infeld equation which arises in electromagnetism, electrostatics and electrodynamics, by employing the Taylor expansion, is dominated by the multiphase differential operator. Last, the $(p,q)$-Laplacian operator, as a special case of double phase operators, can be used to model physical phenomena like non-Newtonian fluids and composites with materials admitting different hardening exponents. (See \cite{BhBF25,BhBF26} for more background on the $(p,q)$-Laplacian operator.)

Recently Bhakta, Biswas and Filippucci \cite{BhBF26} proved Liouville-type theorems for the inequality
\begin{align*}
\laplace_{p} u+\laplace_{q} u+u^{s}\leq 0
\end{align*}
on $\mathbb{R}^{n}$ and its exterior domains. Related results for $(p,q)$-Laplacian equations on domains of $\R^n$ with gradient-dependent nonlinearities were discussed in \cite{BhBF25}. On Riemannian manifolds, concerning the $(p,q)$-Laplacian, Zhao \cite{Zhao26a} obtained Liouville-type results for
	\[ \laplace_{p} u + \laplace_{q} u + V(x)u^{s}\leq 0 \]
	under certain integral conditions on $V(x)$.	
	
Finally, anisotropic Liouville theorems for quasilinear inequalities are also a topic deserving investigation and were studied earlier by D'Ambrosio \cite{DAmb09}. For other Liouville-type theorems concerning harmonic functions on Finsler measure spaces, we refer to \cite{XiaC14, ZhXi14}. More recently, the second-named author \cite{Xion20} established the uniqueness of nonnegative solutions to certain elliptic inequalities on Finsler measure spaces, which is another attempt towards anisotropic Liouville theorems. For interested readers some background regarding geometric analysis on Finsler manifolds can be found in \cite{Ohta21,Ohta17} and references therein, and during the last two decades more and more attention was given to geometric analysis on Finsler manifolds (see, e.g., the monograph \cite{Ohta21} for an exposition). Which results on Riemannian manifolds can be extended to Finsler manifolds arises as a natural and interesting question to be explored. The main differences between Riemannian and Finsler manifolds include the nonlinearity of the gradient operator for functions and the nonsymmetry of the distance function in Finsler manifolds, which have to be carefully addressed.
	

\subsection{The parabolic case}\label{sec1.2}
For the parabolic problem, the studies, though fewer than those for the elliptic problem, have also attracted lots of attention. In the pioneering work \cite{Fuji66} Fujita studied the semilinear heat equation
\begin{align*}
\begin{cases}
&\pt_t u(x,t) =\laplace u(x,t)+ u^{s}(x,t), \quad (x,t)\in \R^{n}\times\R_+,\\
&u(x,0)=u_0(x),\quad x\in \R^n,
\end{cases}
\end{align*}
where $\R_+=(0,\infty)$. He discovered the critical exponent $s_{F}=1+2/n$. More precisely, if $1<s<s_{F}$, any nonnegative global solution must be zero (in other words, all nontrivial nonnegative solutions must blow up in finite time); if $s>s_F$, there exist global nontrivial nonnegative solutions for small enough initial data. The number $s=s_F$ was later known to belong to the blow-up case (see \cite{Haya73} for the cases $n=1$ and $2$, \cite{KoST77} for the general case $n\geq 1$, and also the survey paper \cite{Levi90} on the topic). A systematic investigation of Liouville problems for parabolic equations on Riemannian manifolds was carried out in \cite{Zhan99}. For further related results, we refer the reader to \cite{MaMP17, voGM25, Zhao26b}. In particular, in \cite{MaMP17} the authors consider the parabolic differential inequality
\begin{align*}
\begin{cases}
&\pt_t u(x,t)-\mathrm{div}(|\nabla u(x,t)|^{p-2}\nabla u(x,t))\geq V(x,t)u^s(x,t), \quad (x,t)\in M\times \R_+,\\
&u(x,0)=u_0(x), \quad x\in M,
\end{cases}
\end{align*}
where $M$ is a complete noncompact Riemannian manifold. Under certain hypotheses (similar to HP4 and HP5 below) they obtain Liouville-type results on the nonnegative weak solutions of the above parabolic inequality. In \cite{voGM25}, von Criegern, Grillo, and Monticelli derived Liouville-type results for quite general parabolic inequalities. Last, Zhao \cite{Zhao26b} studied the parabolic counterpart of his elliptic work \cite{Zhao26a}, namely, the case of the parabolic $(p,q)$-Laplacian inequalities. In addition, it is worth mentioning that an extension to hyperbolic space was handled by Bandle, Pozio, and Tesei \cite{BaPT11}.
	
\subsection{Our main results}	Motivated by the above works, in the present paper, we investigate Liouville problems for weak solutions of both elliptic and parabolic partial differential inequalities involving the Finsler $(p,q)$-Laplacian operators with a potential on Finsler measure spaces.
	
	Let $(M^n,F,m)$ be an $n$-dimensional ($n\geq 2$) forward geodesically complete noncompact Finsler manifold with finite reversibility $\finv<\infty$, endowed with a smooth positive measure $m$ (see Section~\ref{sec-pre-finsler} for notions on Finsler measure spaces used in this paper).  We first consider the elliptic partial differential inequality
	\begin{equation}\label{eq-main-PDE-(p,q)}
		\laplace^{m}_{p}u(x)+\laplace^{m}_{q}u(x)+V(x)u^{s}(x)\leq 0,\quad x\in M,
	\end{equation}
	where $p\geq q>1$, $s>p-1$, and the potential $V\in L^{1}_{\mathrm{loc}}(M)$ is positive almost everywhere. For $z\in\set{p,q}$ the Finsler $z$-Laplacian with respect to the measure $m$ is defined as
	\[ \laplace^{m}_{z}u:=\mdiv\sibb{F^{z-2}(\Fgrad u)\Fgrad u}, \]
	where $\Fgrad u$ is the Finsler gradient of the function $u$.
\begin{defn}
A nonnegative function $u\in W^{1,p}_{\mathrm{loc}}(M)$ is a weak solution of \eqref{eq-main-PDE-(p,q)} if for every nonnegative test function $\psi\in W^{1,p}(M)\intersect L^{\infty}(M)$ with compact support, the integral inequality
	\begin{equation}\label{eq-weakSol-PDE-(p,q)}
		\qint{M}{V u^{s}\psi}{m}\leq \qint{M}{F^{p-2}(\Fgrad u)\ud\psi(\Fgrad u)}{m}+\qint{M}{F^{q-2}(\Fgrad u)\ud\psi(\Fgrad u)}{m}
	\end{equation}
 is satisfied.
\end{defn}
	
	We also treat the associated parabolic counterpart,
	\begin{equation}\label{eq-main-PDE-(p,q)-parabolic}
		\case{&\pt_t u(x,t)\geq \laplace^{m}_{p}u(x,t)+\laplace^{m}_{q}u(x,t)+V(x,t)u^{s}(x,t),\quad  (x,t)\in M\times\R_{+},\\
			&u(x,0)=u_{0}(x),\quad x\in M,\\}
	\end{equation}
	where $p\geq q>1$,  $s>\max\set{1,p-1}$, the potential $V\in L^{1}_{\mathrm{loc}}(M\times\R_{+})$ is positive almost everywhere, and $u_0\in L^{1}_{\mathrm{loc}}(M)$ is nonnegative almost everywhere.
	Let $S=M\times\left[0, \infty \right)$.
\begin{defn}
A nonnegative function $u\in W^{1,p}_{\mathrm{loc}}(S)$ is a weak solution of \eqref{eq-main-PDE-(p,q)-parabolic} if for every nonnegative test function $\psi\in W^{1,p}(S)\intersect L^{\infty}(S)$ with compact support, it holds
	\begin{equation}\label{eq-weakSol-PDE-(p,q)-parabolic}
		\begin{aligned}
			&\int_{S} u\,\pt_t\psi \ud m \ud t + \qint{M}{u_{0}\, \psi(x,0)}{m} + \int_{S} V\, u^{s}\,\psi \ud m \ud t\\
			&\leq \int_{S}F^{p-2}(\Fgrad u)\ud\psi(\Fgrad u) \ud m \ud t + \int_{S}F^{q-2}(\Fgrad u)\ud\psi(\Fgrad u) \ud m \ud t.
		\end{aligned}
	\end{equation}
\end{defn}
	
	In the statements below, $r(x):=d(x_{0},x)$ denotes the Finsler distance from a fixed point $x_{0}\in M$ (see Section~\ref{sec-pre-finsler}), and $B_{R}=B_{R}^{+}(x_{0})$ denotes the forward geodesic open ball of radius $R$ with respect to $r$. We will prove that the solution $u$ to \eqref{eq-main-PDE-(p,q)} or \eqref{eq-main-PDE-(p,q)-parabolic} is identically zero almost everywhere provided the potential $V$ satisfies one of the growth conditions given below.
	
	To state our main results, we introduce certain critical exponents. For the elliptic problem and each $z\in\set{p,q}$, we set
	\[ s_{z}:=\frac{sz}{s-z+1}, \qquad \bar{k}_{z}:=\frac{z-1}{s-z+1}. \]
Then we impose the following assumptions.	
\begin{enumerate}
		\item[HP1] For each $z\in\set{p,q}$, there exist positive constants $C_0, C_1, \epsilon_0$ and an exponent $k_z \in [0, \bar{k}_z)$ such that for all sufficiently large $R$ and all $\epsilon \in (0, \epsilon_0)$,
		\[ \qint{\ball[R]\minus\ball[R/2]}{V^{-\bar{k}_{z}+\epsilon}}{m}\leq C_{1} R^{s_{z}+C_{0}\epsilon}\sibb{\ln R}^{k_{z}}. \]
		\item[HP2] For each $z\in\set{p,q}$, there exist positive constants $C_0, C_1, \epsilon_0$ such that for all sufficiently large $R$ and all $\epsilon \in (0, \epsilon_0)$,
		\[ \qint{\ball[R]\minus\ball[R/2]}{V^{-\bar{k}_{z}+\epsilon}}{m}\leq C_{1} R^{s_{z}+C_{0}\epsilon}\sibb{\ln R}^{\bar{k}_{z}},\qquad \qint{\ball[R]\minus\ball[R/2]}{V^{-\bar{k}_{z}-\epsilon}}{m}\leq C_{1} R^{s_{z}+C_{0}\epsilon}\sibb{\ln R}^{\bar{k}_{z}}. \]
		\item[HP3] For each $z\in\set{p,q}$, there exist positive constants $C_{0}, C_{1}, k, \theta, \epsilon_0$, and $\tau_z$ satisfying
		\[ \tau_z>\max\set{\frac{s-z+1}{s}(k+1), 1}, \]
		such that for all sufficiently large $R$ and all $\epsilon\in (0, \epsilon_0)$,
		\[ \qint{\ball[R]\minus\ball[R/2]}{V^{-\bar{k}_{z}+\epsilon}}{m}\leq C_{1} R^{s_{z}+C_{0}\epsilon}\sibb{\ln R}^{k}\ue^{-\epsilon\theta\sibb{\ln R}^{\tau_{z}}}. \]
	\end{enumerate}
\begin{rem}\label{rem1.1}
Our assumptions HP1--HP3 correspond to those in \cite{MaMP15} (see also \cite{Zhao26a}). We refer the interested readers to \cite{MaMP15} for the comparison, the sharpness and further comments on HP1--HP3. In particular, see Remark~1.3 in \cite{MaMP15} for examples of potential functions $V(x)$ satisfying one of HP1--HP3. Here we note that, in HP1 where $k_z$ is less than the critical value $\bar{k}_z$, one inequality is enough; while in HP2 where $k_z$ is relaxed to be $\bar{k}_z$, two inequalities have to be imposed. Moreover, in HP3, the exponent $k$ of $\ln R$ can be larger than $\bar{k}_z$, as long as there appears a compensating term involving $\tau_z$ (larger than $k$ in some way).
\end{rem}
	
	Our first main result is a Liouville-type theorem for the elliptic partial differential inequality \eqref{eq-main-PDE-(p,q)}.
	
	\begin{thm}\label{thm-main-elliptic}
		Let $(M,F,m)$ be a forward geodesically complete noncompact Finsler measure space with finite reversibility $\finv<\infty$. If $u$ is a nonnegative weak solution of \eqref{eq-main-PDE-(p,q)} and the potential $V$ satisfies one of the conditions HP1, HP2 and HP3, then
		\[ u(x)=0 \text{ a.e.\ in } M. \]
	\end{thm}
\begin{rem}
When $(M,F,m)$ is a Riemannian manifold, Theorem~\ref{thm-main-elliptic} reduces to the result in \cite{Zhao26a}, which further covers the $p=q$ case in \cite{MaMP15}. On the other hand, when $p=q$, Theorem~\ref{thm-main-elliptic} reduces to some of results in \cite{Xion20}.
\end{rem}
\begin{rem}
Theorem~\ref{thm-main-elliptic} includes the result on Euclidean space $\R^n$ with an anisotropic function $H$, since $(\R^n, H^*,\lambda)$ is a special Finsler measure space. (A similar remark applies to Theorem~\ref{thm-main-parabolic} below.) Here $H^*$ denotes the dual function of $H$ and $\lambda$ the Lebesgue measure on $\R^n$. The problems on $(\R^n, H^*,\lambda)$ themselves form an interesting topic in geometric analysis; see, e.g., Xia's thesis \cite{XiaC12} for an introduction.
\end{rem}
	
Concerning the parabolic problem, for parameters $\theta_1\geq 1$, $\theta_2\geq 1$ and $R>0$, we introduce the space-time domains
	\[ E_{R,\theta_1,\theta_2}:=\set{(x,t)\in S : t^{\theta_1}+r(x)^{\theta_2}< R^{\theta_2}}. \]
For simplicity of notations, we also use $E_R$ to denote $E_{R,\theta_1,\theta_2}$ whenever no confusion arises.	The relevant parabolic critical indices are
	\[ \bar{s}_{1} :=\frac{s}{s-1}\theta_2,\qquad \bar{s}_{2}:=\frac{1}{s-1},\qquad \bar{l}_{z}:=s_{z}\theta_2=\frac{sz}{s-z+1}\theta_2, \]
	while $\bar{k}_{z}$ remains unchanged.	For the parabolic inequality, we impose the following assumptions.
	
	\begin{enumerate}
		\item[HP4] There exist positive constants $\theta_1\geq 1$, $\theta_2 \geq 1$, $C_0$, $C_1$, $\epsilon_0$ and exponents $s_2 \in [0, \bar{s}_2)$, $k_{z} \in [0, \bar{k}_z)$ such that for all sufficiently large $R$ and all $\epsilon \in (0, \epsilon_0)$,
		\[ \int_{E_{2^{1/\theta_2}R}\minus E_R} t^{(\theta_1-1)(\frac{s}{s-1}-\epsilon)} V^{-\bar{s}_2+\epsilon} \ud m \ud t \leq C_1 R^{\bar{s}_1+C_0\epsilon} \sibb{\ln R}^{s_2}. \]
		For each $z \in \set{p,q}$,
		\[ \int_{E_{2^{1/\theta_2}R}\minus E_R}
		r(x)^{(\theta_2-1)z\sibb{\frac{s}{s-z+1}-\epsilon}} V^{-\bar{k}_z+\epsilon} \ud m \ud t
		\leq C_1 R^{\bar{l}_{z}+C_0\epsilon} \sibb{\ln R}^{k_{z}}. \]
		\item[HP5] There exist positive constants $\theta_1\geq 1$, $\theta_2 \geq 1$, $C_0$, $C_1$, $\epsilon_0$
		such that for all sufficiently large $R$ and all $\epsilon \in (0, \epsilon_0)$,
		\[ \int_{E_{2^{1/\theta_2}R}\minus E_R} t^{(\theta_1-1)(\frac{s}{s-1}-\epsilon)} V^{-\bar{s}_2+\epsilon} \ud m \ud t \leq C_1 R^{\bar{s}_1+C_0\epsilon} \sibb{\ln R}^{\bar{s}_2}. \]
		For each $z \in \set{p,q}$,
		\[ \int_{E_{2^{1/\theta_2}R}\minus E_R} r(x)^{(\theta_2-1)z\sibb{\frac{s}{s-z+1}-\epsilon}} V^{-\bar{k}_z+\epsilon} \ud m \ud t \leq C_1 R^{\bar{l}_{z}+C_0\epsilon} \sibb{\ln R}^{\bar{k}_z}, \]
		\[ \int_{E_{2^{1/\theta_2}R}\minus E_R}
		r(x)^{(\theta_2-1)z\sibb{\frac{s}{s-z+1}+\epsilon}} V^{-\bar{k}_z-\epsilon} \ud m \ud t
		\leq C_1 R^{\bar{l}_{z}+C_0\epsilon} \sibb{\ln R}^{\bar{k}_z}. \]
	\end{enumerate}
\begin{rem}
Our assumptions HP4 and HP5 correspond to those in \cite{MaMP17} (see also \cite{Zhao26b,voGM25}). As mentioned in \cite{MaMP17}, HP4 and HP5 are the counterparts of HP1 and HP2 above, respectively; while there is no known parabolic counterpart of the elliptic assumption HP3. We also note that, in HP4 the indices $s_2$ and $k_z$ are less than the critical ones; while in HP5 the indices are precisely the critical ones and so more inequalities as assumptions are needed. Besides, see Corollaries~5 and 6 in \cite{MaMP17} for examples of potential functions $V(x,t)$ satisfying HP4 and HP5, respectively.
\end{rem}
	
	Our second main result treats the parabolic partial differential inequality \eqref{eq-main-PDE-(p,q)-parabolic}.
	
	\begin{thm}\label{thm-main-parabolic}
		Let $(M,F,m)$ be a forward geodesically complete noncompact Finsler measure space with finite reversibility $\finv<\infty$. If $u$ is a nonnegative weak solution of \eqref{eq-main-PDE-(p,q)-parabolic} and the potential $V$ satisfies either HP4 or HP5, then
		\[ u(x,t)=0  \text{ a.e. in } M\times [0,\infty). \]
	\end{thm}
\begin{rem}
When $(M,F,m)$ is a Riemannian manifold, Theorem~\ref{thm-main-parabolic} reduces to the result in \cite{Zhao26b}, which generalizes the $p=q$ case in \cite{MaMP17}.
\end{rem}

For the proofs of Theorems~\ref{thm-main-elliptic} and \ref{thm-main-parabolic}, we mainly follow the approaches in \cite{MaMP15} (see also \cite{GrSu14,Xion20,Zhao26a}) and \cite{MaMP17} (see also \cite{voGM25,Zhao26b}), respectively. These approaches are essentially based on the nonlinear capacity method by Mitidieri and Pohozaev \cite{MiPo98,MiPo01} mentioned before. Simply speaking, all the proofs consist of two steps. In Step one, we substitute carefully-chosen two-parameter test functions to the weak formulation of the solutions to the elliptic or parabolic inequalities, and derive estimates on the solutions from above in terms of certain quite arbitrary functions. Then in Step two, by choosing appropriate cut-off functions as the above arbitrary functions, we are able to conclude that the weighted (by the potential $V$) integral of some power of the solutions vanishes, which will complete the proofs.
	
	The paper is organized as follows. In Section~\ref{sec-pre}, we collect the necessary background. In Section~\ref{sec-pre-finsler}, we recall the basic concepts and tools of Finsler geometry. In Sections~\ref{sec-pre-inte-elliptic} and \ref{sec-pre-inte-parabolic}, we establish the fundamental integral estimates for weak solutions of the elliptic and parabolic problems, respectively. In	Section~\ref{sec-proof of elliptic} we give the proof of Theorem~\ref{thm-main-elliptic} under one of the three hypotheses HP1--HP3, and in Section~\ref{sec-proof of parabolic} we prove Theorem~\ref{thm-main-parabolic} under the hypothesis HP4 or HP5.
	
	Throughout the paper, the symbol $C$ denotes a generic positive constant depending only on the structural parameters $p$, $q$, $s$, $C_{0}$, $C_{1}$, $k$, $\theta$, $\tau_z$, $\theta_{1}$ and $\theta_{2}$, the Finsler structure $F$, the reversibility constant $\finv$, and the underlying measure $m$. Its value may change from line to line, but it is always independent of the solution $u$ and the variables $x$ and $t$. In addition, whenever we apply Young's inequality, we use the form $ab \le a^p + b^q$ (with $1/p+1/q=1$) instead of the standard one $ab \le a^p/p + b^q/q$.
	
\section{Preliminaries}\label{sec-pre}

In this section we first review some fundamentals on the Finsler measure spaces, and then derive some integral estimates for the solutions to the elliptic and parabolic partial differential inequalities.
\subsection{Finsler Geometry}\label{sec-pre-finsler}
	We recall the basic notions of Finsler geometry that will be used throughout the paper. A thorough treatment can be found in \cite{Ohta21,BaCS00}.
	
	Let $M^n$ be a smooth connected $n$-dimensional ($n\geq 2$) manifold. A Finsler structure on $M$ is a continuous function $F: TM \to [0,\infty)$ that is $C^{\infty}$ on $TM\minus\set{0}$, satisfies $F(x,\lambda V)=\lambda F(x,V)$ for all $\lambda>0$ and $(x,V)\in TM$, and is such that for every non-zero $V\in T_xM$ the fundamental tensor $g_V$ defined by
	\[ g_V(X,Y) = \frac{1}{2}\frac{\pt^2}{\pt s\pt t}\Big|_{s=t=0} F^{2}(x,V+sX+tY), \quad X,Y\in T_xM, \]
	is positive definite.
	The dual Finsler norm on the cotangent bundle is
	\[ F^{*}(x,\xi) = \sup_{V\in T_xM\minus\set{0}} \frac{\xi(V)}{F(x,V)}, \quad \xi\in T^*_xM. \]
	We write $F(V)$ and $F^{*}(\xi)$ when the base point is clear from the context.
	
	The Legendre transformation $\legen: TM \to T^{*}M$ is
	\[ \legen(V)=\case{&g_{V}(V,\cdot), & V \neq 0,\\ &0, & V=0.} \]
	It is a diffeomorphism from $TM\minus\set{0}$ onto $T^{*}M\minus\set{0}$. For a smooth function $\map{u}{M}{\R}$, its Finsler gradient is
	$\Fgrad u = \legen^{-1}(du)$.
	Equivalently, $g_{\Fgrad u}(\Fgrad u, Y)=\ud u(Y)$ for all $Y\in TM$ when $\nabla u\neq 0$. Here we note that the gradient operator is nonlinear, which is different from the Riemannian case.
	
	From the definitions, we have the useful identities
	\[ F(\Fgrad u) = F^{*}(\ud u), \qquad g_{\Fgrad u}(\Fgrad u,\Fgrad u) = F^{2}(\Fgrad u),\quad \text{when } \Fgrad u\neq 0. \]
	Moreover, the following Cauchy--Schwarz inequality (called the fundamental inequality in \cite{BaCS00,Ohta21}) holds for all smooth $u,\phi$:
	\[ \ud\phi(\Fgrad u) \leq F(\Fgrad u)\,F(\Fgrad \phi). \]
	Indeed, if $\Fgrad\phi = 0$, both sides are zero.
	Otherwise, we have
	\[ \ud\phi(\Fgrad u) = g_{\Fgrad\phi}(\Fgrad\phi,\Fgrad u) \leq F(\Fgrad u)\,F(\Fgrad \phi). \]
	The equality is attained if and only if $\Fgrad\phi$ is a nonnegative scalar multiple of $\Fgrad u$.
	
	The length of a piecewise smooth curve $\eta:[0,l]\to M$ is
	\[ L(\eta)=\iint{0}{l}{F(\eta(t),\dot{\eta}(t))}{t}, \]
	and the Finsler distance is $d(x,y) = \inf_\eta L(\eta)$, the infimum taken over all such curves with $\eta(0)=x$, $\eta(l)=y$. Note that generally $d(x,y)\neq d(y,x)$. We assume $(M,F)$ is forward geodesically complete. By the Hopf--Rinow theorem, any two points can be joined by a minimizing geodesic.
	For a fixed $x_0\in M$ we set $r(x) = d(x_0,x)$ and let $B_R = B_R^{+}(x_0)=\set{x\in M : r(x)<R}$ be the forward open geodesic ball. It is easy to check that $F^{*}(dr) = 1$ almost everywhere.
	
	The reversibility of $(M,F)$ is
	\[ \finv:=\sup_{(x,V)\in TM\minus\set{0}}\frac{F(x,V)}{F(x,-V)} \in [1,\infty]. \]
	We assume $\finv<\infty$. Then also $\Lambda_{F^{*}} = \finv$.
	
	Let $m$ be a fixed positive $C^{\infty}$ measure on $M$.
	In local coordinates, let $\ud m = \omega(x)\ud x^1\cdots \ud x^n$. For a smooth vector field $X$, its divergence with respect to $m$ is
	\[ \mdiv X=\frac{1}{\omega}\frac{\pt}{\pt x^{i}}\sibb{\omega X^{i}}. \]
	Then the divergence theorem on $\sibb{M,F,m}$ states that for any smooth vector field $X$ and any $\phi\in C^{\infty}_{c}(M)$, we have
	\[ \qint{M}{\mdiv(X)\phi}{m}=-\qint{M}{\ud \phi(X)}{m}. \]
	
	Because $F$ and any Hilbert norm on $TM$ are locally equivalent, the spaces
	$W^{1,p}_{\mathrm{loc}}(M,m)$ and $L^{p}_{\mathrm{loc}}(M,m)$ are independent of the choices of $F$ and $m$,
	and standard Sobolev embedding theorems apply on coordinate charts.
	
\subsection{Integral estimates for the elliptic inequality}\label{sec-pre-inte-elliptic}
	In this subsection, we establish the key integral estimates for weak solutions of the elliptic inequality \eqref{eq-main-PDE-(p,q)}. Throughout, we write $r(x)=d(x_0,x)$ for the distance from a fixed reference point $x_0\in M$. We begin with a Caccioppoli-type inequality, which provides a fundamental control on the gradient of the solution and the zeroth-order term $Vu^s$ (or its variant $Vu^{s-a}$).
	
	\begin{lem}\label{lem-integralIdentity-cutFunction-young}
		Let $u$ be a nonnegative weak solution of \eqref{eq-main-PDE-(p,q)}. For any positive constants $a$, $b$ satisfying
		\[ 0<a<\min\set{\frac{1}{2}, q-1}, \quad b>\frac{ps}{s-p+1}, \]
		there exists a constant $C(b)> 0$ such that for every Lipschitz function $\phi$ with compact support and $0\leq\phi\leq 1$, the following estimate holds
		\[ \begin{aligned}
			&a\qint{M}{F^{p}(\Fgrad u) u^{-1-a}\phi^{b}}{m}+a\qint{M}{F^{q}(\Fgrad u)u^{-1-a} \phi^{b}}{m}+\qint{M}{V u^{s-a} \phi^{b}}{m}\\
			&\leq C(b)a^{-\frac{(p-1)(s-a)}{s-p+1}}\qint{M}{V^{-\frac{p-a-1}{s-p+1}} F^{\frac{p(s-a)}{s-p+1}}(\Fgrad \phi)}{m}
			+ C(b)a^{-\frac{(q-1)(s-a)}{s-q+1}}\qint{M}{V^{-\frac{q-a-1}{s-q+1}}F^{\frac{q(s-a)}{s-q+1}}(\Fgrad \phi)}{m}.
		\end{aligned} \]
		More explicitly, one may take
		\[ C(b)=\sup_{a\in (0,\min\set{1/2, q-1})}\max\set{b^{\frac{p(s-a)}{s-p+1}}2^{\frac{p(1+s)-a(1+p)-1}{s-p+1}}, b^{\frac{q(s-a)}{s-q+1}}2^{\frac{q(1+s)-a(1+q)-1}{s-q+1}}}. \]
	\end{lem}
	\begin{proof}
		By a standard approximation argument, we may assume that $u$ is strictly positive and $u^{-1} \in L_{\mathrm{loc}}^{\infty}(M)$; otherwise, we may work with $u + \lambda$ (i.e., choosing $\psi=(u+\lambda)^{-a}\phi^b$ below) and let $\lambda \to 0^+$ at the end. Choose the test function $\psi = u^{-a}\phi^{b}$. Its differential is
		\[ \ud\psi=-au^{-a-1}\phi^{b}\ud u+bu^{-a}\phi^{b-1}\ud\phi \quad  \text{a.e. in } M. \]
		Insert $\psi$ into the weak formulation \eqref{eq-weakSol-PDE-(p,q)} to obtain
		\[ \begin{aligned}
			&a\qint{M}{u^{-1-a}\phi^{b}F^{p-2}(\Fgrad u)\ud u(\Fgrad u)}{m} +a\qint{M}{u^{-1-a}\phi^{b}F^{q-2}(\Fgrad u)\ud u(\Fgrad u)}{m}+\qint{M}{Vu^{s-a}\phi^{b}}{m}\\
			&\leq b\qint{M}{u^{-a}\phi^{b-1}F^{p-2}(\Fgrad u) \ud\phi(\Fgrad u)}{m}+b\qint{M}{u^{-a}\phi^{b-1}F^{q-2}(\Fgrad u) \ud\phi(\Fgrad u)}{m}.
		\end{aligned} \]
		Using $\ud u(\Fgrad u)=F^{2}(\Fgrad u)$ and the Cauchy--Schwarz inequality $\ud\phi(\Fgrad u) \leq F(\Fgrad u)\,F(\Fgrad \phi)$, we get
		\begin{equation}\label{eq-inequality-integral-general-temp1}
			\begin{aligned}
				&a\qint{M}{u^{-1-a}\phi^{b}F^{p}(\Fgrad u)}{m}+a\qint{M}{u^{-1-a}\phi^{b}F^{q}(\Fgrad u)}{m}+\qint{M}{Vu^{s-a}\phi^{b}}{m}\\
				&\leq b\qint{M}{u^{-a}\phi^{b-1}F^{p-1}(\Fgrad u)F(\Fgrad\phi)}{m}+b\qint{M}{u^{-a}\phi^{b-1}F^{q-1}(\Fgrad u)F(\Fgrad\phi)}{m}.
			\end{aligned}
		\end{equation}
		
		Now we apply Young's inequality with the exponents $p$ and $p/(p-1)$ to the first term of the right-hand side of \eqref{eq-inequality-integral-general-temp1}. So we obtain
		\[ b\qint{M}{u^{-a}\phi^{b-1}F^{p-1}(\Fgrad u)F(\Fgrad \phi)}{m}\leq \frac{a}{2}\qint{M}{u^{-1-a}\phi^{b}F^{p}(\Fgrad u)}{m}+b^{p} \sibb{\frac{2}{a}}^{p-1}\qint{M}{\phi^{b-p}u^{p-a-1}F^{p}(\Fgrad \phi)}{m}. \]
		An entirely similar computation for the $q$-term yields
		\[ b\qint{M}{u^{-a}\phi^{b-1}F^{q-1}(\Fgrad u)F(\Fgrad \phi)}{m}\leq \frac{a}{2}\qint{M}{u^{-1-a}\phi^{b}F^{q}(\Fgrad u)}{m}+b^{q} \sibb{\frac{2}{a}}^{q-1}\qint{M}{\phi^{b-q}u^{q-a-1}F^{q}(\Fgrad \phi)}{m}. \]
		
		Substituting these two inequalities into \eqref{eq-inequality-integral-general-temp1} leads to
		\begin{equation}\label{eq-inequality-integral-general-temp2}
			\begin{aligned}
				&\frac{a}{2}\qint{M}{u^{-1-a}\phi^{b}F^{p}(\Fgrad u)}{m}+\frac{a}{2}\qint{M}{u^{-1-a}\phi^{b}F^{q}(\Fgrad u)}{m}+\qint{M}{Vu^{s-a}\phi^{b}}{m}\\
				&\leq b^{p} \sibb{\frac{2}{a}}^{p-1}\qint{M}{\phi^{b-p}u^{p-a-1}F^{p}(\Fgrad \phi)}{m}+b^{q} \sibb{\frac{2}{a}}^{q-1}\qint{M}{\phi^{b-q}u^{q-a-1}F^{q}(\Fgrad \phi)}{m}.
			\end{aligned}
		\end{equation}
		
		To eliminate the terms still containing $u$ on the right-hand side, we employ Young's inequality once more. For the $p$-term we use the  exponents $(s-a)/(p-a-1)$ and $(s-a)/(s-p+1)$, which are well-defined because $p-1>a$ and $s>p-1$. This gives
		\[ \begin{aligned}
			&\frac{1}{4} \qint{M}{V u^{s-a} \phi^{b}}{m}+b^{\frac{p(s-a)}{s-p+1}}2^{\frac{(p-1)(2+s)-a(1+p)}{s-p+1}} a^{-\frac{(p-1)(s-a)}{s-p+1}} \qint{M}{\phi^{b-\frac{p(s-a)}{s-p+1}}V^{-\frac{p-a-1}{s-p+1}} F^{\frac{p(s-a)}{s-p+1}}(\Fgrad \phi)}{m}\\
			&\geq b^{p}\sibb{\frac{2}{a}}^{p-1} \qint{M}{\phi^{b-p} u^{p-a-1} F^{p}(\Fgrad \phi)}{m}.
		\end{aligned} \]
		Recalling that $b>ps/(s-p+1)$ and $0 \leq \phi \leq 1$, we can absorb the exponent of $\phi$ into the constant and write the above estimate compactly as
		\[ \begin{aligned}
			&\frac{1}{4} \qint{M}{V u^{s-a} \phi^{b}}{m}+\frac{C(b)}{2} a^{-\frac{(p-1)(s-a)}{s-p+1}} \qint{M}{V^{-\frac{p-a-1}{s-p+1}} F^{\frac{p(s-a)}{s-p+1}}(\Fgrad \phi)}{m}\\
			&\geq b^{p}\sibb{\frac{2}{a}}^{p-1} \qint{M}{\phi^{b-p} u^{p-a-1} F^{p}(\Fgrad \phi)}{m}.
		\end{aligned} \]
		
		For the $q$-term, a similar computation yields
		\[ \begin{aligned}
			&\frac{1}{4} \qint{M}{V u^{s-a} \phi^{b}}{m}+\frac{C(b)}{2} a^{-\frac{(q-1)(s-a)}{s-q+1}} \qint{M}{V^{-\frac{q-a-1}{s-q+1}} F^{\frac{q(s-a)}{s-q+1}}(\Fgrad \phi)}{m}\\
			&\geq b^{q}\sibb{\frac{2}{a}}^{q-1} \qint{M}{\phi^{b-q} u^{q-a-1} F^{q}(\Fgrad \phi)}{m}.
		\end{aligned} \]
		
		Substituting the estimates for both the $p$-term and the $q$-term into \eqref{eq-inequality-integral-general-temp2}, we obtain exactly the desired inequality
		\[ \begin{aligned}
			&a\qint{M}{F^{p}(\Fgrad u) u^{-1-a} \phi^{b}}{m} + a\qint{M}{F^{q}(\Fgrad u) u^{-1-a} \phi^{b}}{m} + \qint{M}{V u^{s-a} \phi^{b}}{m} \\
			&\leq C(b) a^{-\frac{(p-1)(s-a)}{s-p+1}} \qint{M}{V^{-\frac{p-a-1}{s-p+1}} F^{\frac{p(s-a)}{s-p+1}}(\Fgrad \phi)}{m} +C(b) a^{-\frac{(q-1)(s-a)}{s-q+1}} \qint{M}{V^{-\frac{q-a-1}{s-q+1}} F^{\frac{q(s-a)}{s-q+1}}(\Fgrad \phi)}{m},
		\end{aligned} \]
		where the constant $C(b)$ is exactly the one given in the statement of Lemma~\ref{lem-integralIdentity-cutFunction-young}. This completes the proof.
	\end{proof}
	
	From Lemma~\ref{lem-integralIdentity-cutFunction-young} we see that the gradient terms can be controlled by integrals involving only the cut-off function $\phi$ and the potential $V$. The next lemma, which can be viewed as a H\"{o}lder-type estimate, further exploits this kind of controls in order to use the growth hypothesis HP2.
	
	\begin{lem}\label{lem-integralIdentity-cutFunction-holder}
		Let $u$ be a nonnegative weak solution of \eqref{eq-main-PDE-(p,q)}. Assume that the constants $a,b$ satisfy
		\[ 0<a<\min\set{\frac{1}{2}, q-1, \frac{s-p+1}{2(p-1)}}, \quad b>\frac{2ps}{s-p+1}. \]
		Then there exists a constant $C(b)>0$ such that for any $\phi\in\mathrm{Lip}(M)$ with compact support and $0\leq\phi\leq 1$, the following estimate holds
		\[ \begin{aligned}
			\qint{M}{V u^{s}\phi^{b}}{m}&\leq C(b)\sibb{a^{-1}Q}^{\frac{p-1}{p}}\sibb{\qint{M\minus K}{Vu^{s}\phi^{b}}{m}}^{\frac{(a+1)(p-1)}{sp}}J_p^{\frac{s-(a+1)(p-1)}{sp}}\\
			&\quad +C(b)\sibb{a^{-1}Q}^{\frac{q-1}{q}}\sibb{\qint{M\minus K}{Vu^{s}\phi^{b}}{m}}^{\frac{(a+1)(q-1)}{sq}}J_q^{\frac{s-(a+1)(q-1)}{sq}},
		\end{aligned}  \]
		where $K:=\set{x : \phi(x)=1}$, and
		\[ J_p:=\qint{M\minus K}{V^{-\frac{(a+1)(p-1)}{s-(a+1)(p-1)}}F^{\frac{ps}{s-(a+1)(p-1)}}(\Fgrad\phi)}{m},\qquad
		J_q:=\qint{M\minus K}{V^{-\frac{(a+1)(q-1)}{s-(a+1)(q-1)}}F^{\frac{qs}{s-(a+1)(q-1)}}(\Fgrad\phi)}{m}, \]
		\[ Q:=a^{-\frac{(p-1)(s-a)}{s-p+1}} \qint{M}{V^{-\frac{p-a-1}{s-p+1}} F^{\frac{p(s-a)}{s-p+1}}(\Fgrad \phi)}{m}
		+a^{-\frac{(q-1)(s-a)}{s-q+1}} \qint{M}{V^{-\frac{q-a-1}{s-q+1}} F^{\frac{q(s-a)}{s-q+1}}(\Fgrad \phi)}{m}. \]
	\end{lem}
	\begin{proof}
		As in the proof of Lemma~\ref{lem-integralIdentity-cutFunction-young}, we assume that $u>0$. Taking $\psi=\phi^{b}$ as the test function, we have
		\begin{equation}\label{eq-integralIdentity-cutFunction-holder-temp1}
			\begin{aligned}
				\qint{M}{V u^{s}\phi^{b}}{m}&\leq b\qint{M}{F^{p-2}(\Fgrad u)\phi^{b-1}\ud\phi(\Fgrad u)}{m}+b\qint{M}{F^{q-2}(\Fgrad u)\phi^{b-1}\ud\phi(\Fgrad u)}{m}\\
				&\leq b\qint{M}{F^{p-1}(\Fgrad u)\phi^{b-1} F(\Fgrad\phi) }{m}+b\qint{M}{F^{q-1}(\Fgrad u)\phi^{b-1} F(\Fgrad\phi)}{m},
			\end{aligned}
		\end{equation}
		where we used the Cauchy--Schwarz inequality to bound $\ud\phi(\Fgrad u) \leq F(\Fgrad\phi)F(\Fgrad u)$.
		
		We estimate the $p$-term first. By H\"{o}lder's inequality with exponents $p$ and $p/(p-1)$, we get
		\[ \begin{aligned}
			&b\sibb{\qint{M\minus K}{F^{p}(\Fgrad u)\phi^{b}u^{-a-1}}{m}}^{\frac{p-1}{p}}\sibb{\qint{M\minus K}{\phi^{b-p}u^{(a+1)(p-1)}F^{p}(\Fgrad\phi)}{m} }^{\frac{1}{p}}\\
			&\geq b\qint{M}{F^{p-1}(\Fgrad u)\phi^{b-1} F\sibb{\Fgrad\phi} }{m}.
		\end{aligned}  \]
		Notice that on the set where $\phi=1$ the gradient $\Fgrad\phi$ vanishes, so the integration can be restricted to $M\minus K$ without affecting the terms involving $\Fgrad\phi$.
		
		Since $0<a<\min\set{1/2,q-1}$ and $b>ps/(s-p+1)$, we apply Lemma~\ref{lem-integralIdentity-cutFunction-young} with the same $\phi$. This gives
		\[ \qint{M\minus K}{F^{p}(\Fgrad u)\phi^{b}u^{-a-1}}{m}\leq a^{-1}C(b)Q. \]
		
		For the remaining factor containing $u^{(a+1)(p-1)}$, we apply H\"{o}lder's inequality with the pair
		\[ \sibb{\frac{s}{(a+1)(p-1)}, \frac{s}{s-(a+1)(p-1)}}, \]
		which is valid because $a<(s-p+1)/(p-1)$. This gives
		\[ \begin{aligned}
			&\sibb{\qint{M\minus K}{Vu^{s}\phi^{b}}{m}}^{\frac{(a+1)(p-1)}{s}}\sibb{\qint{M\minus K}{V^{-\frac{(a+1)(p-1)}{s-(a+1)(p-1)}}\phi^{b-\frac{ps}{s-(a+1)(p-1)}}F^{\frac{ps}{s-(a+1)(p-1)}}(\Fgrad\phi)}{m}}^{\frac{s-(a+1)(p-1)}{s}}\\
			&\quad\geq\qint{M\minus K}{\phi^{b-p}u^{(a+1)(p-1)}F^{p}(\Fgrad\phi)}{m}.
		\end{aligned} \]
		Since $b>2ps/(s-p+1)$ and $a<(s-p+1)/(2(p-1))$ guarantee $b>ps/(s-(a+1)(p-1))$, we can drop the extra power of $\phi$ inside the left-side integral, and the integral becomes exactly $J_{p}$.
		
		Putting together the estimates, we obtain for the $p$-term
		\[ \begin{aligned}
			b\qint{M}{F^{p-1}(\Fgrad u) \phi^{b-1} F(\Fgrad\phi)}{m}
			\leq C(b) \sibb{a^{-1}Q}^{\frac{p-1}{p}}\sibb{\qint{M\minus K}{V u^{s}\phi^{b}}{m}}^{\frac{(a+1)(p-1)}{sp}} J_{p}^{\frac{s-(a+1)(p-1)}{sp}}.
		\end{aligned} \]
		
		An entirely parallel argument for the $q$-term yields
		\[ \begin{aligned}
			b\qint{M}{F^{q-1}(\Fgrad u) \phi^{b-1} F(\Fgrad\phi)}{m}
			\leq C(b) \sibb{a^{-1}Q}^{\frac{q-1}{q}}\sibb{\qint{M\minus K}{V u^{s}\phi^{b}}{m}}^{\frac{(a+1)(q-1)}{sq}} J_{q}^{\frac{s-(a+1)(q-1)}{sq}}.
		\end{aligned} \]
		Substituting the two bounds into \eqref{eq-integralIdentity-cutFunction-holder-temp1} completes the proof.
	\end{proof}
	
	The hypotheses HP1--HP3 are formulated as integral bounds on the potential $V$ over annular regions.
	The next lemma translates these estimates over annular regions into those over the exterior $M\minus \ball[R]$, which will be useful when we combine them with Lemmas~\ref{lem-integralIdentity-cutFunction-young} and \ref{lem-integralIdentity-cutFunction-holder}.
	
	\begin{lem}\label{lem-inequality-decom-elliptic}
		Let $f \in C^{0}(\R_{+})$ be a nonincreasing nonnegative function.
		\begin{enumerate}
			\item If HP1 holds, then for any sufficiently large $R$ and any $\epsilon \in (0,\epsilon_{0})$,
			\[ \qint{M\minus B_{R}}{f(r(x))V^{-\bar{k}_{z}+\epsilon}}{m}\leq C\iint{\frac{R}{2}}{\infty}{f(r)r^{s_{z}+C_{0}\epsilon-1}\sibb{\ln r}^{k_{z}}}{r}. \]
			\item If HP2 holds, then for any sufficiently large $R$ and any $\epsilon \in (0,\epsilon_{0})$,
			\[ \qint{M\minus B_{R}}{f(r(x))V^{-\bar{k}_{z}+\epsilon}}{m}\leq C\iint{\frac{R}{2}}{\infty}{f(r)r^{s_{z}+C_{0}\epsilon-1}\sibb{\ln r}^{\bar{k}_{z}}}{r}, \]
			and the same estimate holds with $-\bar{k}_{z}-\epsilon$ in place of $-\bar{k}_{z}+\epsilon$,
			\[ \qint{M\minus B_{R}}{f(r(x))V^{-\bar{k}_{z}-\epsilon}}{m}\leq C\iint{\frac{R}{2}}{\infty}{f(r)r^{s_{z}+C_{0}\epsilon-1}\sibb{\ln r}^{\bar{k}_{z}}}{r}. \]
			\item If HP3 holds, then for any sufficiently large $R$ and any $\epsilon \in (0,\epsilon_{0})$,
			\[ \qint{M\minus B_{R}}{f(r(x))V^{-\bar{k}_{z}+\epsilon}}{m}\leq C\iint{\frac{R}{2}}{\infty}{f(r)r^{s_{z}+C_{0}\epsilon-1}\sibb{\ln r}^{k}\ue^{-\epsilon\theta\sibb{\ln r}^{\tau_{z}}}}{r}. \]
		\end{enumerate}
	\end{lem}
	\begin{proof}
		We prove the inequality for HP1 in detail; the other two cases follow by exactly the same reasoning with obvious modifications. The decomposition now uses forward geodesic annuli,
		\[ M\minus B_{R} = \Union_{j=0}^{\infty} \sibb{B_{2^{j+1}R} \minus B_{2^{j}R}}. \]
		
		Since $f$ is nonincreasing and $r(x) \ge 2^{i}R$ on the $i$th annulus, we have
		\[\begin{aligned}
			\qint{M\minus\ball[R]}{f(r) V^{-\bar{k}_{z}+\epsilon}}{m}
			&= \sum_{i=0}^{\infty} \qint{\ball[2^{i+1}R]\minus\ball[2^{i}R]}{f(r) V^{-\bar{k}_{z}+\epsilon}}{m} \\
			&\leq \sum_{i=0}^{\infty} f(2^{i}R) \qint{\ball[2^{i+1}R]\minus\ball[2^{i}R]}{V^{-\bar{k}_{z}+\epsilon}}{m}.
		\end{aligned}\]
		
		By the assumption HP1 applied with $R_i = 2^{i+1}R$, we get
		\[ \qint{\ball[2^{i+1}R]\minus\ball[2^{i}R]}{V^{-\bar{k}_{z}+\epsilon}}{m}
		\leq C_{1} (2^{i+1}R)^{s_{z}+C_{0}\epsilon} \sibb{\ln(2^{i+1}R)}^{k_{z}}. \]
		
		Hence,
		\[ \begin{aligned}
			\qint{M\minus\ball[R]}{f(r) V^{-\bar{k}_{z}+\epsilon}}{m}
			&\leq C_{1} \sum_{i=0}^{\infty} f(2^{i}R) (2^{i+1}R)^{s_{z}+C_{0}\epsilon} \sibb{\ln(2^{i+1}R)}^{k_{z}} \\
			&\leq C \sum_{i=0}^{\infty} f(2^{i}R) (2^{i-1}R)^{s_{z}+C_{0}\epsilon-1} \sibb{\ln(2^{i-1}R)}^{k_{z}}(2^{i-1}R).
		\end{aligned} \]
		The last factor $2^{i-1}R$ is exactly the length of the interval $[2^{i-1}R, 2^{i}R]$. Since $f$ is nonincreasing and the integrand $r^{s_{z}+C_{0}\epsilon-1}(\ln r)^{k_{z}}$ is increasing in $r$, we can bound the sum by the corresponding integral,
		\[ \begin{aligned}
			\sum_{i=0}^{\infty} f(2^{i}R) (2^{i-1}R)^{s_{z}+C_{0}\epsilon-1} (\ln(2^{i-1}R))^{k_{z}} (2^{i-1}R)
			&\leq \sum_{i=0}^{\infty} \iint{2^{i-1}R}{2^{i}R}{f(r) r^{s_{z}+C_{0}\epsilon-1} (\ln r)^{k_{z}}}{r}\\
			&=\iint{\frac{R}{2}}{\infty}{f(r) r^{s_{z}+C_{0}\epsilon-1} (\ln r)^{k_{z}}}{r},
		\end{aligned} \]
		which gives exactly the claimed estimate.
		
		The HP2 variants are proved similarly, with $\bar{k}_{z}$ in place of $k_z$. For HP3, observe that $f(r)\ue^{-\epsilon\theta(\ln r)^{\tau_z}}$ is nonincreasing, and $r^{s_z+C_0\epsilon-1}(\ln r)^k$ is increasing for all sufficiently large $r$. Hence, the same argument yields the desired estimate.
\end{proof}
	
\subsection{Integral estimates for the parabolic inequality}\label{sec-pre-inte-parabolic}
	We now establish the parabolic analogues of the Caccioppoli-type and H\"{o}lder-type estimates proved in the elliptic setting, together with the corresponding integral decomposition lemmas on the space-time domains $E_R$. The method of this part is parallel to that in the elliptic setting, but we must also handle the time derivative of solutions.
	
	We first derive a Caccioppoli-type inequality adapted for the weak formulation of solutions to the parabolic inequality.
	
	\begin{lem}\label{lem-integralIdentity-cutFunction-young-parabolic}
		Let $u$ be a nonnegative weak solution of \eqref{eq-main-PDE-(p,q)-parabolic}. For any positive constants $a$, $b$ satisfying
		\[ 0<a<\min\set{\frac{1}{2}, q-1}, \quad b>\max\set{\frac{s}{s-1}, \frac{ps}{s-p+1}}, \]
		there exists a constant $C(b)>0$ such that for every Lipschitz function $\phi$ with compact support and $0 \leq\phi\leq 1$, the following estimate holds
		\[ \begin{aligned}
			&a\int_{S} u^{-a-1}{\phi}^{b}F^{p}(\Fgrad u) \ud m \ud t+a\int_{S} u^{-a-1}{\phi}^{b}F^{q}(\Fgrad u) \ud m \ud t + \frac{1}{2}\int_{S} Vu^{-a+s}{\phi}^{b} \ud m \ud t\\
			&\leq C(b)a^{-\frac{(p-1)(s-a)}{s-p+1}}\int_{S}F^{\frac{p(s-a)}{s-p+1}} (\Fgrad \phi)V^{-\frac{p-a-1}{s-p+1}} \ud m \ud t +C(b)a^{-\frac{(q-1)(s-a)}{s-q+1}}\int_{S}F^{\frac{q(s-a)}{s-q+1}}(\Fgrad \phi)V^{-\frac{q-a-1}{s-q+1}} \ud m \ud t  \\
			&+ C(b)\int_{S}{\norm{\pt_{t}\phi}}^{\frac{s-a}{s-1}} V^{-\frac{1-a}{s-1}} \ud m \ud t.
		\end{aligned} \]
	\end{lem}
	\begin{proof}
		Similarly to the proof of Lemma~\ref{lem-integralIdentity-cutFunction-young}, we assume that $u>0$. Choose $\psi=u^{-a}\phi^{b}$ as the test function in \eqref{eq-weakSol-PDE-(p,q)-parabolic}. For almost every $(x,t)\in S$, we compute
		\[ \ud\psi = -au^{-1-a}\phi^{b}\ud u + b u^{-a}\phi^{b-1}\ud\phi, \qquad
		\pt_t\psi = -a u^{-1-a}\phi^{b}\pt_t u + b u^{-a}\phi^{b-1}\pt_t\phi. \]
		Insert $\psi$ into the weak formulation \eqref{eq-weakSol-PDE-(p,q)-parabolic} and use $\ud u(\Fgrad u)=F^2(\Fgrad u)$ together with the Cauchy--Schwarz inequality $\ud \phi(\Fgrad u)\leq F(\Fgrad \phi)F(\Fgrad u)$. We obtain
		\begin{equation}\label{eq-inequality-integral-holder-parabolic-temp1}
			\begin{aligned}
				&\int_{S}V u^{s-a}\phi^{b}\ud m\ud t +a\int_{S}F^{p}(\Fgrad u)u^{-a-1}\phi^{b}\ud m\ud t
				+a\int_{S}F^{q}(\Fgrad u)u^{-a-1}\phi^{b}\ud m\ud t\\
				&\leq b\int_{S}F^{p-1}(\Fgrad u)u^{-a}\phi^{b-1}F(\Fgrad\phi)\ud m\ud t+b\int_{S}F^{q-1}(\Fgrad u)u^{-a}\phi^{b-1}F(\Fgrad\phi)\ud m\ud t \\
				&\quad +a\int_{S}u^{-a}\phi^{b}\pt_t u\ud m\ud t-b\int_{S}u^{1-a}\phi^{b-1}\pt_t\phi\ud m\ud t-\int_{M}u_{0}^{1-a}\phi^{b}(x,0)\ud m .
			\end{aligned}
		\end{equation}
		
		For the time derivative term, we apply integration by parts and obtain
		\[ \begin{aligned}
			a\int_{S}u^{-a}\phi^{b}\pt_t u\ud m\ud t&= \frac{a}{1-a}\int_{S}\left(\pt_t(u^{1-a}\phi^{b}) - u^{1-a}\pt_t(\phi^{b})\right)\ud m\ud t \\
			&= -\frac{a}{1-a}\int_{M}u_{0}^{1-a}\phi^{b}(x,0)\ud m
			-\frac{ab}{1-a}\int_{S}u^{1-a}\phi^{b-1}\pt_t\phi\ud m\ud t .
		\end{aligned} \]
		
		Substituting this into \eqref{eq-inequality-integral-holder-parabolic-temp1} and dropping the nonpositive term
		$-a(1-a)^{-1}\int_{M}u_{0}^{1-a}\phi^{b}(x,0)\ud m$ (note $u_0\ge0$) yields
		\begin{equation}\label{eq-inequality-integral-holder-parabolic-temp2}
			\begin{aligned}
				&\int_{S}V u^{s-a}\phi^{b}\ud m\ud t +a\int_{S}F^{p}(\Fgrad u)u^{-a-1}\phi^{b}\ud m\ud t +a\int_{S}F^{q}(\Fgrad u)u^{-a-1}\phi^{b}\ud m\ud t\\
				&\leq b\int_{S}F^{p-1}(\Fgrad u)u^{-a}\phi^{b-1}F(\Fgrad\phi)\ud m\ud t +b\int_{S}F^{q-1}(\Fgrad u)u^{-a}\phi^{b-1}F(\Fgrad\phi)\ud m\ud t\\
				&\quad -\frac{b}{1-a}\int_{S}u^{1-a}\phi^{b-1}\pt_t\phi\ud m\ud t.
			\end{aligned}
		\end{equation}
		
		Now apply Young's inequality with the exponents $p$ and $p/(p-1)$ to the $p$-term on the right-hand side of \eqref{eq-inequality-integral-holder-parabolic-temp2}. We have
		\[ \begin{aligned}
			&b\int_{S}F^{p-1}(\Fgrad u)u^{-a}\phi^{b-1}F(\Fgrad\phi)\ud m\ud t \\
			&\leq \frac{a}{2}\int_{S}F^{p}(\Fgrad u)u^{-a-1}\phi^{b}\ud m\ud t
			+ b^{p}\sibb{\frac{a}{2}}^{-(p-1)} \int_{S}u^{p-a-1}\phi^{b-p}F^{p}(\Fgrad\phi)\ud m\ud t.
		\end{aligned} \]
		An entirely similar computation for the $q$-term yields
		\[ \begin{aligned}
			&b\int_{S}F^{q-1}(\Fgrad u)u^{-a}\phi^{b-1}F(\Fgrad\phi)\ud m\ud t \\
			&\leq \frac{a}{2}\int_{S}F^{q}(\Fgrad u)u^{-a-1}\phi^{b}\ud m\ud t
			+ b^{q}\sibb{\frac{a}{2}}^{-(q-1)}\int_{S}u^{q-a-1}\phi^{b-q}F^{q}(\Fgrad\phi)\ud m\ud t.
		\end{aligned} \]

		To eliminate the terms still containing $u$ on the right-hand side, we employ Young's inequality once more. For the $p$-term we use the exponents $(s-a)/(s-p+1)$ and $(s-a)/(p-a-1)$, which gives
		\[ \begin{aligned}
			&b^{p}\sibb{\frac{a}{2}}^{-(p-1)}\int_{S}u^{p-a-1}\phi^{b-p}F^{p}(\Fgrad\phi)\ud m\ud t\\
			&\leq C(b) a^{-\frac{(p-1)(s-a)}{s-p+1}} \int_{S}F^{\frac{p(s-a)}{s-p+1}}(\Fgrad\phi)V^{-\frac{p-a-1}{s-p+1}} \phi^{b-\frac{p(s-a)}{s-p+1}}\ud m\ud t
			+\frac{1}{4}\int_{S}V u^{s-a}\phi^{b}\ud m\ud t .
		\end{aligned} \]
		An analogous inequality holds for the $q$-term
		\[ \begin{aligned}
			&b^{q}\sibb{\frac{a}{2}}^{-(q-1)}\int_{S}u^{q-a-1}\phi^{b-q}F^{q}(\Fgrad\phi)\ud m\ud t \\
			&\leq C(b) a^{-\frac{(q-1)(s-a)}{s-q+1}} \int_{S}F^{\frac{q(s-a)}{s-q+1}}(\Fgrad\phi)V^{-\frac{q-a-1}{s-q+1}} \phi^{b-\frac{q(s-a)}{s-q+1}}\ud m\ud t
			+\frac{1}{4}\int_{S}V u^{s-a}\phi^{b}\ud m\ud t .
		\end{aligned} \]
		
		Since $a<1/2$, for the term containing $\partial_t \phi$, we have
		\[ \frac{b}{1-a}\int_{S}u^{1-a}\phi^{b-1}|\pt_{t}\phi|\ud m\ud t
		\leq \frac{1}{4}\int_{S}V u^{s-a}\phi^{b}\ud m\ud t
		+ C(b)\int_{S}V^{-\frac{1-a}{s-1}}|\pt_{t}\phi|^{\frac{s-a}{s-1}} \phi^{b-\frac{s-a}{s-1}} \ud m\ud t, \]
		by the Young's inequality with the exponents $(s-a)/(1-a)$ and $(s-a)/(s-1)$.
		
		Substituting these estimates back into \eqref{eq-inequality-integral-holder-parabolic-temp2} and recalling that $0\leq\phi\leq 1$ and
		\[ b>\max\set{{\frac{s}{s-1}, \frac{ps}{s-p+1}}} \]
		so as to absorb any extra powers of $\phi$, we
		finally arrive at
		\[ \begin{aligned}
			&\frac{a}{2}\int_{S}F^{p}(\Fgrad u)u^{-a-1}\phi^{b}\ud m\ud t
			+\frac{a}{2}\int_{S}F^{q}(\Fgrad u)u^{-a-1}\phi^{b}\ud m\ud t
			+\frac{1}{4}\int_{S}V u^{s-a}\phi^{b}\ud m\ud t \\
			\leq& C(b) a^{-\frac{(p-1)(s-a)}{s-p+1}}\int_{S}F^{\frac{p(s-a)}{s-p+1}}(\Fgrad\phi)V^{-\frac{p-a-1}{s-p+1}}\ud m\ud t+ C(b) a^{-\frac{(q-1)(s-a)}{s-q+1}}\int_{S}F^{\frac{q(s-a)}{s-q+1}}(\Fgrad\phi)V^{-\frac{q-a-1}{s-q+1}}\ud m\ud t \\
			&+ C(b)\int_{S}|\pt_{t}\phi|^{\frac{s-a}{s-1}}V^{-\frac{1-a}{s-1}}\ud m\ud t ,
		\end{aligned} \]
		which is exactly the desired inequality. This completes the proof.
	\end{proof}
	
	In analogy with the elliptic situation, we also need a parabolic H\"{o}lder-type estimate that allows us to relate the integral of $Vu^s$ over a space-time domain $K=E_R$ to that over its complement $S\minus K$ in the later proof.
	
	\begin{lem}\label{lem-integralIdentity-cutFunction-holder-parabolic}
		Let $u$ be a nonnegative weak solution of \eqref{eq-main-PDE-(p,q)-parabolic}. Assume that the constants $a,b>0$ satisfy
		\[ a <\min\set{\frac{1}{2}, q-1, \frac{s-p+1}{2(p-1)}}, \quad b>\max\set{\frac{s}{s-1}, \frac{2ps}{s-p+1}}. \]
		Then there exists a constant $C(b)>0$ such that for any $\phi\in\mathrm{Lip}\sibb{S}$ with compact support and $0\leq\phi\leq 1$, the following estimate holds
		\[ \begin{aligned}
			\int_{S} V u^{s}\phi^{b}\ud m\ud t &\leq C(b)(a^{-1}Q)^{\frac{p-1}{p}}\sibb{\int_{S\minus K} V u^{s}\phi^{b}\ud m\ud t}^{\frac{(a+1)(p-1)}{sp}} J_p^{\frac{s-(a+1)(p-1)}{sp}} \\
			&\quad + C(b)(a^{-1}Q)^{\frac{q-1}{q}}\sibb{\int_{S\minus K} V u^{s}\phi^{b}\ud m\ud t}^{\frac{(a+1)(q-1)}{sq}} J_q^{\frac{s-(a+1)(q-1)}{sq}}\\
			&\quad + C(b)\sibb{\int_{S\minus K} V u^{s}\phi^{b}\ud m\ud t}^{\frac{1}{s}} J_t^{\frac{s-1}{s}},
		\end{aligned} \]
		where $K:=\set{(x,t)\in S : \phi(x,t)=1}$,
		\[ \begin{aligned}
			Q&:=a^{-\frac{(p-1)(s-a)}{s-p+1}}\int_{S} V^{-\frac{p-a-1}{s-p+1}} F^{\frac{p(s-a)}{s-p+1}}(\Fgrad\phi) \ud m \ud t + a^{-\frac{(q-1)(s-a)}{s-q+1}}\int_{S} V^{-\frac{q-a-1}{s-q+1}} F^{\frac{q(s-a)}{s-q+1}}(\Fgrad\phi) \ud m \ud t\\
			&\quad + \int_{S} V^{-\frac{1-a}{s-1}} |\pt_t \phi|^{\frac{s-a}{s-1}} \ud m \ud t,\\
			J_z &:= \int_{S\minus K} V^{-\frac{(a+1)(z-1)}{s-(a+1)(z-1)}} F^{\frac{zs}{s-(a+1)(z-1)}}(\Fgrad\phi) \ud m \ud t
			\text{ for each } z\in\set{p,q},
		\end{aligned} \]
and
\begin{align*}
			J_t &:= \int_{S\minus K} V^{-\frac{1}{s-1}} |\pt_t \phi|^{\frac{s}{s-1}} \ud m \ud t.
\end{align*}
	\end{lem}
	\begin{proof}
		As in the proof of the previous lemmas, we may assume $u>0$. Taking $\phi^{b}$ as the test function in \eqref{eq-weakSol-PDE-(p,q)-parabolic} and using the Cauchy--Schwarz inequality, we obtain
		\begin{equation}\label{eq-parabolic-holder-temp1}
			\begin{aligned}
				\int_{S}Vu^{s}\phi^{b}\ud m\ud t &\leq b\int_{S}F^{p-1}(\Fgrad u)\phi^{b-1}F(\Fgrad\phi)\ud m\ud t+ b\int_{S}F^{q-1}(\Fgrad u)\phi^{b-1}F(\Fgrad\phi)\ud m\ud t \\
				&\quad + b\int_{S}u\phi^{b-1}|\pt_t\phi|\ud m\ud t -\int_{M}u_{0}\phi^{b}(x,0)\ud m .
			\end{aligned}
		\end{equation}
		The last term is nonpositive and can be dropped.
		
		We estimate the $p$-term first. By H\"{o}lder's inequality with exponents $p$ and $p/(p-1)$, we get
		\[ \begin{aligned}
			&b\int_{S}F^{p-1}(\Fgrad u)\phi^{b-1}F(\Fgrad\phi)\ud m\ud t\\
			&\leq b\sibb{\int_{S\minus K}F^{p}(\Fgrad u)\phi^{b}u^{-a-1}\ud m\ud t}^{\frac{p-1}{p}}
			\sibb{\int_{S\minus K}\phi^{b-p}u^{(a+1)(p-1)}F^{p}(\Fgrad\phi)\ud m\ud t}^{\frac{1}{p}}.
		\end{aligned} \]
		Notice that the gradient $\Fgrad\phi$ vanishes on $K$, so the integration can be restricted to $S\minus K$ without affecting the terms involving $\Fgrad\phi$.
		
		Since $a,b$ satisfy the conditions of Lemma~\ref{lem-integralIdentity-cutFunction-young-parabolic}, applying this lemma with the same $\phi$ yields
		\[ \int_{S\minus K}F^{p}(\Fgrad u)\phi^{b}u^{-a-1}\ud m\ud t \leq a^{-1}C(b)Q. \]
		
		For the remaining factor containing $u^{(a+1)(p-1)}$, we use H\"{o}lder's inequality with the pair
		\[ \sibb{\frac{s}{(a+1)(p-1)},\frac s{s-(a+1)(p-1)}}, \]
		which is admissible because $a<(s-p+1)/(2(p-1))$ guarantees $s>(a+1)(p-1)$.  This gives
		\[ \begin{aligned}
			&\sibb{\int_{S\minus K}Vu^{s}\phi^{b}\ud m\ud t}^{\frac{(a+1)(p-1)}{s}}
			\sibb{\int_{S\minus K}V^{-\frac{(a+1)(p-1)}{s-(a+1)(p-1)}}\phi^{b-\frac{ps}{s-(a+1)(p-1)}}F^{\frac{ps}{s-(a+1)(p-1)}}(\Fgrad\phi)\ud m\ud t}^{\frac{s-(a+1)(p-1)}{s}}\\
			&\geq\int_{S\minus K}\phi^{b-p}u^{(a+1)(p-1)}F^{p}(\Fgrad\phi)\ud m\ud t .
		\end{aligned} \]
		Because $b>ps/(s-(a+1)(p-1))$ and $0\leq\phi\leq1$, the extra powers of $\phi$ inside the left-hand integral can be dropped, and the integral becomes exactly $J_{p}$.
		
		Putting these estimates together, we obtain for the $p$-term,
		\[ \begin{aligned}
			b\int_{S}F^{p-1}(\Fgrad u)\phi^{b-1}F(\Fgrad\phi)\ud m\ud t
			\leq C(b)(a^{-1}Q)^{\frac{p-1}{p}}\sibb{\int_{S\minus K}Vu^{s}\phi^{b}\ud m\ud t}^{\frac{(a+1)(p-1)}{sp}} J_{p}^{\frac{s-(a+1)(p-1)}{sp}} .
		\end{aligned} \]
		
		An entirely parallel argument for the $q$-term gives
		\[ \begin{aligned}
			b\int_{S}F^{q-1}(\Fgrad u)\phi^{b-1}F(\Fgrad\phi)\ud m\ud t
			\leq C(b)(a^{-1}Q)^{\frac{q-1}{q}}\sibb{\int_{S\minus K}Vu^{s}\phi^{b}\ud m\ud t}^{\frac{(a+1)(q-1)}{sq}} J_{q}^{\frac{s-(a+1)(q-1)}{sq}} .
		\end{aligned} \]
		
		Finally, for the term containing $\partial_t \phi$, we apply H\"{o}lder's inequality with exponents $(s,s/(s-1))$ to get
		\[ \begin{aligned}
			b\int_{S}u\phi^{b-1}|\pt_t\phi|\ud m\ud t\leq b\sibb{\int_{S\minus K}Vu^{s}\phi^{b}\ud m\ud t}^{\frac{1}{s}}
			\sibb{\int_{S\minus K}V^{-\frac{1}{s-1}}\phi^{b-\frac{s}{s-1}}|\pt_t\phi|^{\frac{s}{s-1}}\ud m\ud t}^{\frac{s-1}{s}}.
		\end{aligned} \]
		Again, the conditions $b>s/(s-1)$ and $0\leq\phi\leq1$ allow us to absorb the remaining powers of $\phi$ into the constant, so the integral becomes $J_{t}$.  Consequently,
		\[ b\int_{S}u\phi^{b-1}|\pt_t\phi|\ud m\ud t
		\leq C(b)\sibb{\int_{S\minus K}Vu^{s}\phi^{b}\ud m\ud t}^{\frac{1}{s}} J_{t}^{\frac{s-1}{s}}. \]
		
		Substituting these estimates into \eqref{eq-parabolic-holder-temp1}, we finish the proof.
	\end{proof}
	
	Finally, the parabolic growth conditions HP4 and HP5 are expressed by integrals over annuli generated by the space-time domains $E_R$. The following two lemmas convert these conditions into estimates over the exterior $S\minus E_R$ weighted by a nonincreasing function $f$, exactly as was done in the elliptic case.
	
	\begin{lem}\label{lem-inequality-decom-HP4}
		Let $f \in C^{0}(\R_{+})$ be a nonincreasing nonnegative function. For $z\in\set{p,q}$, assume that the condition HP4 holds.
		\begin{enumerate}
			\item For any $R$ sufficiently large and any $\epsilon \in (0,\epsilon_0)$,
			\[ \begin{aligned}
				\int_{S \minus E_R}f([r^{\theta_2}(x) +t^{\theta_1}]^{\frac{1}{\theta_2}})t^{(\theta_1-1)
					\sibb{\frac{s}{s-1}-\epsilon }}V^{-\bar{s}_{2} + \epsilon} \ud m \ud t
				\leq C_{1}\int_{R/2^{\frac{1}{\theta_2}}}^{\infty } f(r)r^{\bar{s}_1 +C_0\epsilon -1}\sibb{\ln r}^{s_2} \ud r.
			\end{aligned} \]
			\item For each $z \in \set{p,q}$, any $R$ sufficiently large and any $\epsilon \in (0,\epsilon_0)$,
			\[ \begin{aligned}
				\int_{S \minus E_R}f([r^{\theta_2}(x) +t^{\theta_1}]^{\frac{1}{\theta_2}})r(x)^{(\theta_2-1)z\sibb{\frac{s}{s-z+1}-\epsilon}}V^{-\bar{k}_{z} + \epsilon} \ud m \ud t
				\leq C_{1}\int_{R/2^{\frac{1}{\theta_2}}}^{\infty } f(r)r^{\bar{l}_{z} +C_0\epsilon -1}\sibb{\ln r}^{k_{z}} \ud r.
			\end{aligned} \]
		\end{enumerate}
	\end{lem}
	\begin{proof}
		We prove the second inequality in detail; the proof of the first one follows by exactly the same reasoning.
		
		Consider a fixed $z\in\set{p,q}$. We decompose the region $S \minus E_R$ into the union of annuli
		\[ S \minus E_R = \Union_{j=0}^{\infty} \sibb{E_{2^{(j+1)/\theta_2}R} \minus E_{2^{j/\theta_2}R}}. \]
		Let $\rho(x,t) = [r^{\theta_2}(x) +t^{\theta_1}]^{1/\theta_2}$. Since $f$ is nonincreasing and $\rho(x,t) \geq 2^{j/\theta_2}R$ on the $j$th annulus, we have
		\begin{align}\label{inequality in lemma2.6}
			&\int_{S \minus E_R} f(\rho) r^{(\theta_2-1)z\sibb{\frac{s}{s-z+1}-\epsilon}} V^{-\bar{k}_z+\epsilon} \ud m \ud t  \nonumber \\
			&= \sum_{j=0}^{\infty} \int_{E_{2^{(j+1)/\theta_2}R} \minus E_{2^{j/\theta_2}R}} f(\rho) r^{(\theta_2-1)z\sibb{\frac{s}{s-z+1}-\epsilon}} V^{-\bar{k}_z+\epsilon} \ud m \ud t  \nonumber \\
			&\leq \sum_{j=0}^{\infty} f\sibb{2^{j/\theta_2}R} \int_{E_{2^{(j+1)/\theta_2}R} \minus E_{2^{j/\theta_2}R}} r^{(\theta_2-1)z\sibb{\frac{s}{s-z+1}-\epsilon}} V^{-\bar{k}_z+\epsilon} \ud m \ud t .
		\end{align}
		
		By the assumption HP4 applied with $R_j = 2^{j/\theta_2}R$ and $\epsilon$, we have
		\[ \begin{aligned}
			\text{RHS of }\eqref{inequality in lemma2.6}
			&\leq C_{1} \sum_{j=0}^{\infty} f\sibb{2^{j/\theta_2}R} \sibb{2^{(j-1)/\theta_2}R}^{\bar{l}_z+C_0\epsilon-1} \sibb{\ln\sibb{2^{(j-1)/\theta_2}R}}^{k_{z}}\Delta r_j ,
		\end{aligned} \]
		where $\Delta r_j := 2^{(j-1)/\theta_2}R \cdot (2^{1/\theta_2}-1)$ and we have adjusted the constant and exponents implicitly.
		
		Since $f$ is nonincreasing and the integrand $r^{\bar{l}_z+C_0\epsilon-1}\sibb{\ln r}^{k_{z}}$ is increasing in $r$, we can bound the sum by the corresponding integral,
		\[ \begin{aligned}
			\sum_{j=0}^{\infty} f\sibb{2^{j/\theta_2}R} \sibb{2^{(j-1)/\theta_2}R}^{\bar{l}_z+C_0\epsilon-1} \sibb{\ln\sibb{2^{(j-1)/\theta_2}R}}^{k_{z}}\Delta r_j
			\leq \int_{R/2^{\frac{1}{\theta_2}}}^{\infty}{f(r) r^{\bar{l}_z+C_0\epsilon-1} \sibb{\ln r}^{k_{z}}} \ud r,
		\end{aligned} \]
		which completes the proof.
	\end{proof}
	
	\begin{lem}\label{lem-inequality-decom-HP5}
		Let $f \in C^{0}(\R_{+})$ be a nonincreasing nonnegative function. For $z\in\set{p,q}$, assume that the condition HP5 holds.
		\begin{enumerate}
			\item For any $R$ sufficiently large and any $\epsilon \in (0,\epsilon_0)$,
			\[ \int_{S \minus E_R} f([r^{\theta_2}(x) +t^{\theta_1}]^{\frac{1}{\theta_2}})t^{(\theta_1-1)\sibb{\frac{s}{s-1}-\epsilon}}V^{-\bar{s}_2 + \epsilon} \ud m \ud t
			\leq C_{1}\int_{R/2^{\frac{1}{\theta_2}}}^{\infty } f(r)r^{\bar{s}_1 +C_0\epsilon -1}\sibb{\ln r}^{\bar{s}_2} \ud r. \]
			\item For each $z \in \set{p,q}$, any $R$ sufficiently large and any $\epsilon \in (0,\epsilon_0)$,
			\[ \begin{aligned}
				\int_{S \minus E_R}f([r^{\theta_2}(x) +t^{\theta_1}]^{\frac{1}{\theta_2}})r(x)^{(\theta_2-1)z\sibb{\frac{s}{s-z+1}-\epsilon}}V^{-\bar{k}_{z} + \epsilon} \ud m \ud t
				\leq C_{1}\int_{R/2^{\frac{1}{\theta_2}}}^{\infty } f(r)r^{\bar{l}_{z} +C_0\epsilon -1}\sibb{\ln r}^{\bar{k}_{z}} \ud r, \\
				\int_{S \minus E_R}f([r^{\theta_2}(x) +t^{\theta_1}]^{\frac{1}{\theta_2}})r(x)^{(\theta_2-1)z\sibb{\frac{s}{s-z+1}+\epsilon }}V^{-\bar{k}_{z} - \epsilon} \ud m \ud t
				\leq C_{1}\int_{R/2^{\frac{1}{\theta_2}}}^{\infty } f(r)r^{\bar{l}_{z} +C_0\epsilon -1}\sibb{\ln r}^{\bar{k}_{z}} \ud r.
			\end{aligned} \]
		\end{enumerate}
	\end{lem}
	\begin{proof}
		The proof is almost the same as that of Lemma~\ref{lem-inequality-decom-HP4}, the only difference being the form of the right-hand side in the hypothesis HP5.
	\end{proof}
	
\section{The Proof of Theorem~\ref{thm-main-elliptic}}\label{sec-proof of elliptic}
	We now prove the elliptic Liouville theorem. The three hypotheses HP1--HP3 are treated separately, but they share a common structure. Similarly to \cite{MaMP15,Zhao26a}, we construct a family of cut-off functions with a slowly decaying tail, and combine the Caccioppoli-type (or H\"{o}lder-type) estimates of Section~\ref{sec-pre-inte-elliptic} together with the decomposition lemmas to obtain integral bounds on the solution that will force it to vanish eventually.
\subsection{The case of HP1}\label{sec-the case of HP1}
	We apply the Caccioppoli-type inequality (Lemma~\ref{lem-integralIdentity-cutFunction-young}) with a cut-off function that decays as a small negative power of the distance function in order to force $\int_{M} V u^{s} dm= 0$.
	
	Fix a constant $b$ with $b>ps/\sibb{s-p+1}$. For every sufficiently large $R$, define
	\[ a:=\frac{1}{\ln R}, \]
	so that $a<1/2$ (for instance, for $R> e^{2}$). For each integer $j\geq 2$, consider the cut-off function $\phi_j = \phi \eta_j$, where
	\[ \phi(x) =\case{&1, \qquad  &r(x) \leq R,\\
		&\sibb{\frac{r(x)}{R}}^{-A a}, \qquad  &r(x)>R,}\qquad \eta_j(x)=\case{&1, \qquad &r(x) < jR,\\
		&2-\frac{r(x)}{jR}, \qquad &jR \leq r(x) < 2jR,\\
		&0, \qquad &r(x) \geq 2jR.} \]
	The constant $A$ is chosen large enough so that
	\[ A>\frac{2(C_{0}+s+1)}{qs}, \]
	where $C_{0}$ is the constant appearing in HP1. Since $s>p-1>0$, by taking $R$ sufficiently large we may assume $a<s/2$, and hence
	\[ A>\frac{C_{0}+s+1}{q(s-a)}. \]
	
	For every $k\geq 0$, the triangle inequality gives
	\[ F^{k}(\Fgrad\phi_j) \leq C\sibb{F^{k}(\Fgrad\phi)+\phi^{k} F^{k}(\Fgrad\eta_j)}. \]
	Moreover, $F^{*}(\ud r)=F(\Fgrad r)=1$ almost everywhere implies that
	\[ F(\Fgrad\eta_j)=F^{*}(\ud \eta_j)\leq\case{&0, \qquad &r<jR\text{ or }r\geq 2jR ,\\ &\frac{\finv}{jR},\qquad &jR\leq r<2jR.} \]
	Similarly, we have
	\[ F(\Fgrad\phi)\leq\case{&0, \qquad &r<R,\\ &\finv Aar^{-Aa-1}R^{Aa},\qquad &r\geq R.}  \]
	
	Applying Lemma~\ref{lem-integralIdentity-cutFunction-young} with the cut-off function $\phi_j$ and the constants $a$, $b$, we obtain
	\begin{equation}\label{eq-inequality-integral-HP1-main-temp1}
		\begin{aligned}
			\qint{M}{V u^{s-a} \phi_j^{b}}{m} &\leq C a^{-\frac{(p-1)(s-a)}{s-p+1}} \qint{M}{V^{-\frac{p-a-1}{s-p+1}} F^{\frac{p(s-a)}{s-p+1}}(\Fgrad\phi_j)}{m}\\
			&\quad + C a^{-\frac{(q-1)(s-a)}{s-q+1}}
			\qint{M}{V^{-\frac{q-a-1}{s-q+1}} F^{\frac{q(s-a)}{s-q+1}}(\Fgrad\phi_j)}{m}\\
			&\leq Ca^{-\frac{(p-1)(s-a)}{s-p+1}}\sibb{I_{p,1} + I_{p,2}}+Ca^{-\frac{(q-1)(s-a)}{s-q+1}} \sibb{I_{q,1} + I_{q,2}},
		\end{aligned}
	\end{equation}
	where
	\[ \begin{aligned}
		I_{p,1} &= \qint{M}{V^{-\frac{p-a-1}{s-p+1}} F^{\frac{p(s-a)}{s-p+1}}(\Fgrad\phi)}{m}, &
		I_{p,2} &= \qint{M}{V^{-\frac{p-a-1}{s-p+1}} \phi^{\frac{p(s-a)}{s-p+1}} F^{\frac{p(s-a)}{s-p+1}}(\Fgrad\eta_j)}{m}, \\
		I_{q,1} &= \qint{M}{V^{-\frac{q-a-1}{s-q+1}} F^{\frac{q(s-a)}{s-q+1}}(\Fgrad\phi)}{m}, &
		I_{q,2} &= \qint{M}{V^{-\frac{q-a-1}{s-q+1}} \phi^{\frac{q(s-a)}{s-q+1}} F^{\frac{q(s-a)}{s-q+1}}(\Fgrad\eta_j)}{m}.
	\end{aligned} \]
	
	We now estimate these four integrals, beginning with $I_{p,1}$. From the definition of $\phi$ and the fact that $R^{a}=e$, we get
	\[\begin{aligned}
		I_{p,1}
		&\leq C a^{\frac{p(s-a)}{s-p+1}}
		\qint{M \minus \ball[R]}{V^{-\frac{p-a-1}{s-p+1}} r(x)^{-(A a + 1)\frac{p(s-a)}{s-p+1}}}{m}.
	\end{aligned}\]
	Notice that
\begin{align*}
-\frac{p-a-1}{s-p+1}=-\bar{k}_{p}+\frac{a}{s-p+1}.
\end{align*}
Hence, by Lemma~\ref{lem-inequality-decom-elliptic} (the HP1 case) with $\epsilon=a/\sibb{s-p+1}$,
	\[ I_{p,1}\leq C a^{\frac{p(s-a)}{s-p+1}}\iint{\frac{R}{2}}{\infty}{r^{-(A a + 1)\frac{p(s-a)}{s-p+1}+s_{p}+C_{0}\frac{a}{s-p+1}-1}\sibb{\ln r}^{k_{p}} }{r}, \]
	where $s_{p} = sp/\sibb{s-p+1}$, and $k_{p}<\bar{k}_{p}=\sibb{p-1}/\sibb{s-p+1}$ is the exponent from HP1.
	
	Set
	\[ \alpha=(A a + 1)\frac{p(s-a)}{s-p+1}-s_{p}-C_{0}\frac{a}{s-p+1}. \]
	Because $A>\sibb{C_{0}+s+1}/\sibb{p(s-a)}$, one checks that $\alpha > a > 0$ for all small $a$. The change of variables $t = \alpha \ln r$ for $\alpha>0$ yields
	\[ \begin{aligned}
		I_{p,1} &\leq C a^{\frac{p(s-a)}{s-p+1}}\iint{\alpha\ln\frac{R}{2}}{\infty}{\alpha^{-1}\exp\sibb{\frac{t}{\alpha} \sibb{-(A a + 1)\frac{p(s-a)}{s-p+1}+s_{p}+C_{0}\frac{a}{s-p+1}}} \sibb{\frac{t}{\alpha}}^{k_{p}} }{t}\\
		&\leq C a^{\frac{p(s-a)}{s-p+1}}\iint{0}{\infty}{\alpha^{-1-k_{p}}\ue^{-t} t^{k_{p}}}{t}\leq Ca^{\frac{p(s-a)}{s-p+1}}\alpha^{-1-k_{p}}\leq C a^{\frac{p(s-a)}{s-p+1}-k_{p}-1}.
	\end{aligned} \]
	
	Turning to $I_{p,2}$, we observe that on the annulus $\ball[2jR]\setminus\ball[jR]$ one has $\phi \leq j^{-A a}$ and $F(\nabla\eta_j)\leq \finv/(jR)$.  Consequently,
	\[ \begin{aligned}
		I_{p,2}&\leq j^{-Aa\frac{p(s-a)}{s-p+1}}
		\sibb{\frac{\finv}{jR}}^{\frac{p(s-a)}{s-p+1}}
		\qint{\ball[2jR]\minus\ball[jR]}{V^{-\bar{k}_{p}+\frac{a}{s-p+1}}}{m}\\
		&\leq C j^{-\sibb{Aa+1}\frac{p(s-a)}{s-p+1}} R^{-\frac{p(s-a)}{s-p+1}} (2jR)^{s_{p}+C_{0}\frac{a}{s-p+1} }\sibb{\ln(2jR)}^{k_{p}}.
	\end{aligned} \]
	
	Using $R^{a}=e$, the factor involving $R$ simplifies to a pure constant,
	\[ R^{-\frac{p(s-a)}{s-p+1}+s_{p}+C_{0}\frac{a}{s-p+1}}= R^{\frac{a(p+C_{0})}{s-p+1}}=\exp\sibb{\frac{p+C_{0}}{s-p+1}}. \]
	Thus, we have
	\[ I_{p,2} \leq Cj^{-\alpha}\sibb{\ln(2jR)}^{k_{p}}\leq Cj^{-a}\sibb{\ln(2jR)}^{k_{p}}. \]
	
	The quantities $I_{q,1}$ and $I_{q,2}$ are handled in exactly the same way, simply replacing $p$ by $q$ and $k_{p}$ by $k_{q}$ ($<\bar{k}_{q}$) throughout.  We obtain
	\[ I_{q,1}\leq Ca^{\frac{q(s-a)}{s-q+1}-k_{q}-1},\quad I_{q,2}\leq Cj^{-a}\sibb{\ln(2jR)}^{k_{q}}. \]
	
	Now we insert these bounds into \eqref{eq-inequality-integral-HP1-main-temp1} to get
	\[ \begin{aligned}
		\qint{M}{V u^{s-a}\phi_{j}^{b}}{m}&\leq C\sibb{a^{-\frac{(p-1)(s-a)}{s-p+1}}j^{-a}\sibb{\ln(2jR)}^{k_{p}}+a^{-\frac{(q-1)(s-a)}{s-q+1}}j^{-a}\sibb{\ln(2jR)}^{k_{q}}} \\
		&\quad + C\sibb{a^{-\frac{a}{s-p+1}+\bar{k}_{p}-k_{p}}+a^{-\frac{a}{s-q+1}+\bar{k}_{q}-k_{q}}}.
	\end{aligned} \]
	
	Letting $j \to \infty$, we obtain
	\[ \qint{\ball[R]}{V u^{s-a}}{m}
	\leq C\sibb{a^{-\frac{a}{s-p+1}+\bar{k}_{p}-k_{p}}+a^{-\frac{a}{s-q+1}+\bar{k}_{q}-k_{q}}}. \]
	
	Because of $k_z<\bar{k}_z$, $z\in\set{p,q}$, the exponents of $a$ are positive if $a$ is small enough. Fix $b>ps/\sibb{s-p+1}$ and let $R \to \infty$. The right-hand side tends to zero, and by Fatou's lemma,
	\[ \qint{M}{V u^{s}}{m} = 0. \]
	Since $V>0$ almost everywhere and $u\geq 0$, we conclude $u = 0$ a.e. in $M$.
	
\subsection{The case of HP2}\label{sec-the case of HP2}
	In this case, we need the sharper H\"{o}lder-type estimate (Lemma~\ref{lem-integralIdentity-cutFunction-holder}) instead of the Caccioppoli-type inequality. The proof first shows that $Vu^s$ is globally integrable, and then uses the fact that the integral over the complement $M\minus B_R$ tends to zero to force that $\int_M V u^sdm=0$.
	
	As in the HP1 case, we fix a constant $b$ with $b>2ps/\sibb{s-p+1}$. For all sufficiently large $R$, the corresponding $a:=1/\ln R$ satisfies the conditions required in Lemma~\ref{lem-integralIdentity-cutFunction-holder}, so the lemma applies with this fixed $b$ and the varying $a$.
	
	Define the cut-off function $\phi_j=\phi\eta_j$ exactly as in Section~\ref{sec-the case of HP1}.
	The exponent $A$ is now chosen large enough such that
	\[ A>\max_{z\in\set{p,q}}\set{\frac{2\sibb{C_{0}+s+1}}{zs},\frac{2\sibb{s+1}}{zs}+\frac{4C_{0}(z-1)}{z(s-z+1)}}. \]
	Since $s>p-1>0$, by taking $R$ sufficiently large we may assume $a<\min\set{(s-p+1)/(2(p-1)),1/2,s/2}$. Then we have
	\[ A>\max_{z\in\set{p,q}}\set{\frac{C_{0}+s+1}{z(s-a)},\frac{s+1}{z(s-a)}+\frac{C_{0}s(z-1)}{z(s-a)(s-(a+1)(z-1))}}. \]
	
	With this choice, the parameters $a,b$ satisfy the hypotheses of Lemma~\ref{lem-integralIdentity-cutFunction-holder} as long as $R$ is so large that
\begin{align*}
a< \min\set{\frac{1}{2},q-1,\frac{s-p+1}{2(p-1)}}.
\end{align*}
Applying that lemma with $\phi_j$ gives
	\begin{equation}\label{eq-inequality-integral-HP2-main-temp1}
		\begin{aligned}
			\qint{M}{V u^{s}\phi_j^{b}}{m} &\leq C\sibb{a^{-1} Q_j}^{\frac{p-1}{p}} \sibb{\qint{M\minus K}{V u^{s} \phi_j^{b}}{m}}^{\frac{(a+1)(p-1)}{sp}} J_{p}^{\frac{s-(a+1)(p-1)}{sp}} \\
			&\quad+C \sibb{a^{-1} Q_j}^{\frac{q-1}{q}}
			\sibb{\qint{M\minus K}{V u^{s} \phi_j^{b}}{m}}^{\frac{(a+1)(q-1)}{sq}} J_{q}^{\frac{s-(a+1)(q-1)}{sq}},
		\end{aligned}
	\end{equation}
	where $K=\set{x : \phi_j(x)=1}$,
	\[ Q_{j}:=a^{-\frac{(p-1)(s-a)}{s-p+1}} \qint{M}{V^{-\frac{p-a-1}{s-p+1}} F^{\frac{p(s-a)}{s-p+1}}(\Fgrad \phi_{j})}{m}
	+a^{-\frac{(q-1)(s-a)}{s-q+1}} \qint{M}{V^{-\frac{q-a-1}{s-q+1}} F^{\frac{q(s-a)}{s-q+1}}(\Fgrad \phi_{j})}{m}, \]
	\[ J_p:=\qint{M\minus K}{V^{-\frac{(a+1)(p-1)}{s-(a+1)(p-1)}}F^{\frac{ps}{s-(a+1)(p-1)}}(\Fgrad\phi_{j})}{m}, \]
and
\begin{align*}
J_q:=\qint{M\minus K}{V^{-\frac{(a+1)(q-1)}{s-(a+1)(q-1)}}F^{\frac{qs}{s-(a+1)(q-1)}}(\Fgrad\phi_{j})}{m}.
\end{align*}
	
	First, we estimate the quantity $Q_{j}$. Since $A>\sibb{C_{0}+s+1}/\sibb{p(s-a)}$, by comparing with the estimates for $I_{p,1},I_{p,2}$ in the HP1 case, the only difference now is to use the HP2 case of Lemma~\ref{lem-inequality-decom-elliptic} (the inequalities with $\bar{k}_{z}$ instead of $k_{z}$). Repeating the same computation, we obtain
	\begin{equation}\label{eq-inequality-integral-HP2-Q_j-temp1}
		\begin{aligned}
			Q_{j}&\leq C\sibb{a^{-\frac{(p-1)(s-a)}{s-p+1}}j^{-a}\sibb{\ln(2jR)}^{\bar{k}_{p}}+a^{-\frac{(q-1)(s-a)}{s-q+1}}j^{-a}\sibb{\ln(2jR)}^{\bar{k}_{q}}} \\
			&\quad + C\sibb{a^{-\frac{a}{s-p+1}}+a^{-\frac{a}{s-q+1}}}.
		\end{aligned}
	\end{equation}
	Since $0<a<1$, we have $\ue^{-1/\ue}\leq a^{a}< 1$. The inequality~\eqref{eq-inequality-integral-HP2-Q_j-temp1} implies
	\[ Q_{j}\leq C+C\sibb{a^{-\frac{(p-1)(s-a)}{s-p+1}}j^{-a}\sibb{\ln(2jR)}^{\bar{k}_{p}}+a^{-\frac{(q-1)(s-a)}{s-q+1}}j^{-a}\sibb{\ln(2jR)}^{\bar{k}_{q}}}. \]
	Moreover, the terms containing $j$ vanish as $j\to\infty$.  Hence
	\[ \limsup_{j\to\infty} Q_j \leq C. \]
	
	Next, we estimate $J_{p}$ and $J_{q}$. For $J_{p}$, we split it by use of the pointwise bound
	\[F^{k}(\Fgrad\phi_j) \leq C\sibb{F^{k}(\Fgrad\phi) + \phi^{k}F^{k}(\Fgrad\eta_j)},\]
	with $k>0$. From the explicit expression of $\phi$ and the fact that $R^{a}=e$, we then obtain
	\[ J_{p}\leq C(J_{p,1}+J_{p,2}), \]
	where
\begin{align*}
		J_{p,1}&=\qint{M}{V^{-\frac{(a+1)(p-1)}{s-(a+1)(p-1)}} F^{\frac{ps}{s-(a+1)(p-1)}}(\Fgrad\phi)}{m}\\
		&\leq Ca^{\frac{ps}{s-(a+1)(p-1)}}\qint{M\minus\ball[R]}{V^{-\frac{(a+1)(p-1)}{s-(a+1)(p-1)}}r^{-(Aa+1)\frac{ps}{s-(a+1)(p-1)}}}{m},
\end{align*}
and
\begin{align*}
J_{p,2}&=\qint{M}{V^{-\frac{(a+1)(p-1)}{s-(a+1)(p-1)}} \phi^{\frac{ps}{s-(a+1)(p-1)}}F^{\frac{ps}{s-(a+1)(p-1)}}(\Fgrad\eta_j)}{m}.
\end{align*}
	
	Set
	\[ \epsilon_{p} := \frac{sa(p-1)}{(s-p+1)(s-(p-1)(a+1))}. \]
	Applying the HP2 case of Lemma~\ref{lem-inequality-decom-elliptic} with $\epsilon$ replaced by $\epsilon_{p}$ yields
	\[ \begin{aligned}
		J_{p,1}&\leq Ca^{\frac{ps}{s-(a+1)(p-1)}}\iint{\frac{R}{2}}{\infty}{r^{-(Aa+1)\frac{ps}{s-(a+1)(p-1)}+s_{p}+C_{0}\epsilon_{p}-1}\sibb{\ln r}^{\bar{k}_{p}}}{r}.
	\end{aligned} \]
	
	Let
	\[ \alpha=(Aa+1)\frac{ps}{s-(a+1)(p-1)}-s_{p}-C_{0}\epsilon_{p}. \]
	Thanks to the choice of $A$, we have
	\[ \alpha> (Aa+1)\frac{p(s-a)}{s-p+1}-s_{p}-C_{0}\epsilon_{p}>a>0. \]
	Using the change of variables $t = \alpha\ln r$, we get
	\[ \begin{aligned}
		J_{p,1} &\leq Ca^{\frac{ps}{s-(a+1)(p-1)}}\iint{\alpha\ln\frac{R}{2}}{\infty}{\alpha^{-1}\exp\sibb{\frac{t}{\alpha} \sibb{-(Aa+1)\frac{ps}{s-(a+1)(p-1)}+s_{p}+C_{0}\epsilon_{p}}} \sibb{\frac{t}{\alpha}}^{\bar{k}_{p}}}{t}\\
		&\leq Ca^{\frac{ps}{s-(a+1)(p-1)}}\iint{0}{\infty}{\alpha^{-1-\bar{k}_{p}}\ue^{-t}t^{\bar{k}_{p}}}{t}\\
		&\leq Ca^{\frac{ps}{s-(a+1)(p-1)}-\bar{k}_{p}-1}.
	\end{aligned} \]
	
	For $J_{p,2}$, on $\ball[2jR]\minus\ball[jR]$ we have $\phi\leq j^{-Aa}$ and $F(\Fgrad\eta_j) \leq \finv/(jR)$.  Hence, by HP2,
	\[ \begin{aligned}
		J_{p,2}&\leq C\qint{\ball[2jR]\minus\ball[jR]}{V^{-\frac{(a+1)(p-1)}{s-(a+1)(p-1)}} j^{-Aa\frac{ps}{s-(a+1)(p-1)}}\sibb{\frac{\finv}{jR}}^{\frac{ps}{s-(a+1)(p-1)}}}{m}\\
		&\leq C j^{-(Aa+1)\frac{ps}{s-(a+1)(p-1)}}R^{-\frac{ps}{s-(a+1)(p-1)}}(2jR)^{s_{p}+C_{0}\epsilon_{p}} \sibb{\ln(2jR)}^{\bar{k}_{p}}.
	\end{aligned} \]
	Using $R^{a}=e$ again, the powers of $R$ combine to
	\[ R^{-\frac{ps}{s-(a+1)(p-1)}+s_{p}+C_{0}\epsilon_{p}}=R^{\frac{(C_{0}-p)sa(p-1)}{(s-p+1)(s-(p-1)(a+1))}}\leq C. \]
	Thus,
	\[ J_{p,2}\leq Cj^{-a}\sibb{\ln(2jR)}^{\bar{k}_{p}}, \]
	where we used
\begin{align*}
\alpha=\frac{ps(Aa+1)}{s-(a+1)(p-1)}-s_{p}-C_{0}\epsilon_{p}>a.
\end{align*}
	
	Combining the two parts, we obtain
	\[ J_{p}\leq Ca^{\frac{ps}{s-(a+1)(p-1)}-\bar{k}_{p}-1}+Cj^{-a}\sibb{\ln(2jR)}^{\bar{k}_{p}}. \]
	Exactly the same reasoning for the $q$-term gives
	\[ J_{q}\leq Ca^{\frac{qs}{s-(a+1)(q-1)}-\bar{k}_{q}-1}+Cj^{-a}\sibb{\ln(2jR)}^{\bar{k}_{q}}. \]
	
	Inserting the estimates for $J_{p}$, $J_{q}$, and $Q_{j}$ into \eqref{eq-inequality-integral-HP2-main-temp1} yields
	\begin{equation}\label{eq-inequality-integral-HP2-temp2}
		\begin{aligned}
			\qint{M}{V u^{s}\phi_j^{b}}{m} &\leq C\sibb{a^{\frac{a (p-1)}{p (s-p+1)}}+a^{-\frac{p-1}{p}}j^{-\chi_{p}}\sibb{\ln(2jR)}^{\delta_{p}}} \sibb{\qint{M\minus K}{V u^{s} \phi_j^{b}}{m}}^{\frac{(a+1)(p-1)}{sp}}\\
			&\quad+C\sibb{a^{\frac{a (q-1)}{q(s-q+1)}}+a^{-\frac{q-1}{q}}j^{-\chi_{q}}\sibb{\ln(2jR)}^{\delta_{q}}}
			\sibb{\qint{M\minus K}{V u^{s} \phi_j^{b}}{m}}^{\frac{(a+1)(q-1)}{sq}},
		\end{aligned}
	\end{equation}
	where for each $z\in\set{p,q}$, we define
	\[ \delta_{z}:=\bar{k}_{z}\frac{s-(a+1)(z-1)}{sz},\]
and
\[\chi_{z}:=\frac{a(s-(a+1)(z-1))}{sz}. \]
	
	The integral inequality \eqref{eq-inequality-integral-HP2-temp2} remains valid (possibly with a different constant $C$) if we replace the three integrals with $1+\int_M V u^{s}\phi_j^{b}dm$. Now set
	\[ \gamma = \max\set{\frac{(a+1)(p-1)}{sp},\frac{(a+1)(q-1)}{sq}}>0. \]
	If $a$ is small enough, then $\gamma< 1$. This fact combined with $V u^{s}\in L_{\mathrm{loc}}^{1}(M)$ implies
	\[ \begin{aligned}
		\sibb{1+\qint{M}{V u^{s}\phi_j^{b}}{m}}^{1-\gamma}&\leq C\sibb{a^{\frac{a(p-1)}{p (s-p+1)}}+a^{-\frac{p-1}{p}}j^{-\chi_p}\sibb{\ln(2jR)}^{\delta_{p}}}\\
		&\quad+C\sibb{a^{\frac{a(q-1)}{q(s-q+1)}}+a^{-\frac{q-1}{q}}j^{-\chi_q}\sibb{\ln(2jR)}^{\delta_{q}}}.
	\end{aligned} \]
	Letting $j\to\infty$ and using that $\chi_{p}>0$, $\chi_{q}>0$ for all sufficiently large $R$, we deduce, for any sufficiently large $R$,
	\[ \qint{\ball[R]}{V u^{s}}{m}\leq C. \]
	Thus, we may conclude that $V u^{s}$ is integrable in $M$.
	
	Letting $j\to\infty$ in the inequality~\eqref{eq-inequality-integral-HP2-temp2}, we deduce (note that $K = \ball[R]$ for every $j$)
	\begin{equation}\label{eq-inequality-integral-HP2-temp3}
		\qint{\ball[R]}{V u^{s}}{m}
		\leq C \sibb{\qint{M\minus\ball[R]}{V u^{s}}{m}}^{\frac{(a+1)(p-1)}{sp}}+C \sibb{\qint{M\minus\ball[R]}{V u^{s}}{m}}^{\frac{(a+1)(q-1)}{sq}}.
	\end{equation}
	
	Finally, let $R\to\infty$ in \eqref{eq-inequality-integral-HP2-temp3}. The integral over $M\minus\ball[R]$ tends to zero because the whole integral $\int_M V u^{s}$ is finite. The right-hand side therefore tends to zero, while the left-hand side increases to $\int_{M} V u^{s}$. Hence
	\[ \qint{M}{V u^{s}}{m} = 0,\]
	which forces $u=0$ a.e. in $M$.
	
\subsection{The case of HP3}
	The proof follows the same strategy as in HP1, but now the hypothesis HP3 provides an extra exponentially decaying factor in the integral estimates.
	
	As in the previous cases, we fix a constant $b$ with $b>ps/\sibb{s-p+1}$ and, for all large $R$, set
	\[ a := \frac{1}{\ln R}. \]
	The cut-off function is exactly the same used in the cases of HP1 and HP2, $\phi_j(x) := \phi(x) \eta_j(x)$. We choose the constant $A$ sufficiently large so that
	\[ A>\frac{2\sibb{C_{0}+s+1}}{qs}, \]
	where $C_0$ is taken from the condition HP3. Since $s>p-1>0$, by taking $R$ sufficiently large we may assume $a<s/2$, and hence
	\[ A>\frac{C_{0}+s+1}{q(s-a)}. \]
	
	Applying Lemma~\ref{lem-integralIdentity-cutFunction-young} with the cut-off function $\phi_j$ and the constants $a,b$, we obtain
	\begin{equation}\label{eq-inequality-integral-HP3-main-temp1}
		\begin{aligned}
			\qint{M}{V u^{s-a} \phi_j^{b}}{m}&\leq C a^{-\frac{(p-1)(s-a)}{s-p+1}} \qint{M}{V^{-\frac{p-a-1}{s-p+1}} F^{\frac{p(s-a)}{s-p+1}}(\Fgrad\phi_j)}{m} \\
			&\quad + Ca^{-\frac{(q-1)(s-a)}{s-q+1}} \qint{M}{V^{-\frac{q-a-1}{s-q+1}} F^{\frac{q(s-a)}{s-q+1}}(\Fgrad\phi_j)}{m} \\
			&\leq Ca^{-\frac{(p-1)(s-a)}{s-p+1}}\sibb{I_{p,1} + I_{p,2}} + Ca^{-\frac{(q-1)(s-a)}{s-q+1}}\sibb{I_{q,1} + I_{q,2}},
		\end{aligned}
	\end{equation}
	where $I_{p,1}, I_{p,2}, I_{q,1}$, $I_{q,2}$ are defined exactly as in Section~\ref{sec-the case of HP1}.
	
	Let us first estimate $I_{p,1}$. From the definition of $\phi$, the relation $R^a = e$, and Lemma~\ref{lem-inequality-decom-elliptic} with
	\[ \epsilon_{p} = \frac{a}{s-p+1}=-\frac{p-a-1}{s-p+1}+\frac{p-1}{s-p+1}, \]
	we get
	\[ \begin{aligned}
		I_{p,1}&\leq C a^{\frac{p(s-a)}{s-p+1}}
		\qint{M \minus \ball[R]}{V^{-\bar{k}_{p}+\frac{a}{s-p+1}} r(x)^{-(Aa+1)\frac{p(s-a)}{s-p+1}}}{m} \\
		&\leq C a^{\frac{p(s-a)}{s-p+1}}
		\iint{\frac{R}{2}}{\infty}{r^{-(Aa+1)\frac{p(s-a)}{s-p+1}+s_{p}+C_{0}\frac{a}{s-p+1}-1}
			(\ln r)^{k} \ue^{-\frac{a\theta}{s-p+1}(\ln r)^{\tau_{p}}}}{r}.
	\end{aligned} \]
	Note that
\begin{align*}
A&>\frac{C_{0}+s+1}{q(s-a)}\geq\frac{C_{0}+s+1}{p(s-a)}
\end{align*}
implies
	\[ -(Aa+1)\frac{p(s-a)}{s-p+1}+s_{p}+C_{0}\frac{a}{s-p+1}<-a<0. \]
	Setting
	\[ t = \sibb{\frac{a\theta}{s-p+1}}^{\frac{1}{\tau_{p}}}\ln r, \]
	we transform the integral into
	\[ \begin{aligned}
		I_{p,1} &\leq C a^{\frac{p(s-a)}{s-p+1}}
		\iint{0}{\infty}{\exp\left(\frac{t}{\sibb{\frac{a\theta}{s-p+1}}^{\frac{1}{\tau_{p}}}}
				\sibb{-(Aa+1)\tfrac{p(s-a)}{s-p+1}+s_{p}+C_{0}\tfrac{a}{s-p+1}}\right)
			\frac{t^{k}}{\sibb{\frac{a\theta}{s-p+1}}^{\frac{k+1}{\tau_{p}}}}\ue^{-t^{\tau_{p}}}}{t}\\
		&\leq Ca^{\frac{p(s-a)}{s-p+1}}
		\iint{0}{\infty}{t^{k} \sibb{\frac{a\theta}{s-p+1}}^{-\frac{k+1}{\tau_{p}}} \ue^{-t^{\tau_{p}}}}{t}\\
		&\leq Ca^{\frac{p(s-a)}{s-p+1}-\frac{k+1}{\tau_{p}}}.
	\end{aligned} \]
	
	The term $I_{p,2}$ is handled exactly as in the HP1 case. On $\ball[2jR]\minus\ball[jR]$, we have $\phi \leq j^{-Aa}$ and $F(\nabla\eta_j) \leq \finv/\sibb{jR}$. Using HP3 with the same $\epsilon = a/\sibb{s-p+1}$, we get
	\[ \begin{aligned}
		I_{p,2}&\leq j^{-Aa\frac{p(s-a)}{s-p+1}} \sibb{\frac{\finv}{jR}}^{\frac{p(s-a)}{s-p+1}}
		\qint{\ball[2jR]\minus\ball[jR]}{V^{-\bar{k}_{p}+\frac{a}{s-p+1}}}{m} \\
		&\leq C j^{-(Aa+1)\frac{p(s-a)}{s-p+1}} R^{-\frac{p(s-a)}{s-p+1}}
		(2jR)^{s_{p}+C_{0}\frac{a}{s-p+1}} (\ln(2jR))^{k}
		\ue^{-\frac{a\theta}{s-p+1}(\ln(2jR))^{\tau_{p}}}.
	\end{aligned} \]
	Using $R^a = e$ and collecting the powers of $j$, we obtain
	\[ I_{p,2} \leq C j^{-\alpha}R^{\frac{a(C_{0}+p)}{s-p+1}} (\ln(2jR))^{k}\leq C j^{-a}(\ln(2jR))^{k}, \]
	where $$\alpha=(Aa+1)\frac{p(s-a)}{s-p+1} - s_{p} - C_{0}\frac{a}{s-p+1}>a> 0$$ for the choice of $A$.
	
	Completely analogous arguments give
	\[ I_{q,1}\leq C a^{\frac{q(s-a)}{s-q+1}-\frac{k+1}{\tau_{q}}}, \quad  I_{q,2}\leq C j^{-a} (\ln(2jR))^{k}. \]
	
	Substituting these estimates into \eqref{eq-inequality-integral-HP3-main-temp1} yields
	\[ \begin{aligned}
		\qint{M}{V u^{s-a} \phi_j^{b}}{m}
		&\leq C\sibb{a^{-\frac{(p-1)(s-a)}{s-p+1}}I_{p,1}+a^{-\frac{(p-1)(s-a)}{s-p+1}}I_{p,2}+a^{-\frac{(q-1)(s-a)}{s-q+1}} I_{q,1}+a^{-\frac{(q-1)(s-a)}{s-q+1}}I_{q,2}} \\
		&\leq C\sibb{a^{\frac{s-a}{s-p+1}-\frac{k+1}{\tau_{p}}}+j^{-a} a^{-\frac{(p-1)(s-a)}{s-p+1}}\sibb{\ln(2jR)}^{k}}\\
		&\quad +C\sibb{a^{\frac{s-a}{s-q+1}-\frac{k+1}{\tau_{q}}}+j^{-a} a^{-\frac{(q-1)(s-a)}{s-q+1}}\sibb{\ln(2jR)}^{k}}.
	\end{aligned} \]
	
	Letting $j\to\infty$, we see that the terms containing $j$ vanish because $a > 0$. Observing that $\phi_j = 1$ on $\ball[R]$, we get
\begin{align}\label{eq-3.7}
\qint{\ball[R]}{V u^{s-a}}{m}
	\leq C\sibb{a^{\frac{s-a}{s-p+1}-\frac{k+1}{\tau_{p}}}+a^{\frac{s-a}{s-q+1}-\frac{k+1}{\tau_{q}}}}.
\end{align}
Since $\tau_z>\sibb{s-z+1}(k+1)/s$ for $z=p,q$ by HP3, we have
	\[ \frac{s}{s-z+1}-\frac{k+1}{\tau_z}>0.\]
	Because $a = 1/\ln R \to 0^+$ as $R\to\infty$, the exponent
\begin{align*}
\frac{s-a}{s-z+1}-\frac{k+1}{\tau_z}
\end{align*}
converges to the positive limit above. Hence, it is strictly positive for any sufficiently large $R$.
	
	Finally, let $R\to\infty$ so that $a\to 0^{+}$. The right-hand side of \eqref{eq-3.7} tends to zero, and so
	\[ \qint{M}{V u^{s}}{m} = 0, \]
	which implies that $u = 0$ a.e. on $M$.

\section{The Proof of Theorem~\ref{thm-main-parabolic}}\label{sec-proof of parabolic}
	The proof in the parabolic case follows a similar strategy. We construct suitable cut-off functions involving both spatial and temporal variables, apply the parabolic Caccioppoli-type or H\"{o}lder-type estimates from Section~\ref{sec-pre-inte-parabolic}, and then combine them with the decomposition lemmas corresponding to HP4 or HP5.
\subsection{The case of HP4}
	The proof follows the same strategy as in HP1, but with an extra time derivative term that can be handled well in the parabolic Caccioppoli-type inequality (Lemma~\ref{lem-integralIdentity-cutFunction-young-parabolic}).
	
	First we choose a constant $b$ independent of $R$ such that
	\[ b>\max\set{\frac{s}{s-1},\frac{sp}{s-p+1}}, \]
	and let $a:=1/\ln R$ for every sufficiently large $R$. Hence, the assumptions of Lemma~\ref{lem-integralIdentity-cutFunction-young-parabolic} are satisfied with this fixed $b$ and the varying $a$, and the constant $C(b)$ is an absolute constant.
	
	For each integer $j\geq 2$, define $\phi_j=\phi\eta_j$ with
	\[ \phi(x,t)=\case{&1,\quad &\text{if}\quad &(x,t)\in E_R,\\
		&\sibb{\frac{r(x)^{\theta_2} + t^{\theta_1}}{R^{\theta_2}}}^{-A a},\quad &\text{if}\quad &(x,t)\in S\minus E_R,} \]
and
	\[ \eta_j(x,t)=\case{&1,\quad &\text{if}\quad &(x,t)\in E_{jR},\\
		&2-\frac{r(x)^{\theta_2} + t^{\theta_1}}{(jR)^{\theta_2}},\quad &\text{if}\quad &(x,t)\in E_{2^{\frac{1}{\theta_2}}jR} \minus E_{jR},\\
		& 0,\quad &\text{if}\quad &(x,t)\in S\minus E_{2^{\frac{1}{\theta_2}}jR}.} \]
	The constant $A$ is chosen sufficiently large such that
	\[ A>\max_{z\in\set{p,q}}\set{\frac{2\sibb{\theta_2 z+C_0 +1}}{\theta_2zs}, \frac{2\sibb{\theta_2+C_0 +1}}{\theta_2s}}, \]
	where $C_{0}$ is the constant appearing in HP4. Since $a<1/2$ and $s>1$, we have
	\[ A>\max_{z\in\set{p,q}}\set{\frac{\theta_2 z+C_0 +1}{\theta_2z(s-a)}, \frac{\theta_2+C_0 +1}{\theta_2(s-a)}}. \]
	
	As in the elliptic case, choosing $\phi_j$ as the test function in Lemma~\ref{lem-integralIdentity-cutFunction-young-parabolic} yields
	\begin{equation}\label{eq-inequality-integral-HP4-main-temp1}
		\begin{aligned}
			\int_{S} V u^{-a+s}\phi_j^{b} \ud m\ud t & \leq Ca^{-\frac{(p-1)(s-a)}{s-p+1}}\int_{S} F(\Fgrad \phi_j)^{\frac{p(s-a)}{s-p+1}} V^{-\frac{p-a-1}{s-p+1}} \ud m \ud t\\
			&\quad+Ca^{-\frac{(q-1)(s-a)}{s-q+1}}\int_{S}  F(\Fgrad \phi_j)^{\frac{q(s-a)}{s-q+1}} V^{-\frac{q-a-1}{s-q+1}} \ud m \ud t\\
			&\quad+C\int_{S}{\norm{\pt_{t}\phi_j}}^{\frac{s-a}{s-1}} V^{-\frac{1-a}{s-1}} \ud m \ud t\\
			&\leq C(I_{p,1} + I_{p,2} + I_{q,1} + I_{q,2}+I_{t,1} + I_{t,2}),
		\end{aligned}
	\end{equation}
	where
	\[ \begin{aligned}
		I_{p,1}&=a^{-\frac{(p-1)(s-a)}{s-p+1}}\int_{S\minus E_R} F^{\frac{p(s-a)}{s-p+1}}(\Fgrad\phi) V^{-\frac{p-a-1}{s-p+1}}\ud m\ud t,\\
		I_{p,2}&=a^{-\frac{(p-1)(s-a)}{s-p+1}}\int_{E_{2^{\frac{1}{\theta_2}}jR}\minus E_{jR}} \phi^{\frac{p(s-a)}{s-p+1}} F^{\frac{p(s-a)}{s-p+1}}(\Fgrad\eta_j) V^{-\frac{p-a-1}{s-p+1}}\ud m\ud t,\\
		I_{q,1}&=a^{-\frac{(q-1)(s-a)}{s-q+1}}\int_{S\minus E_R} F^{\frac{q(s-a)}{s-q+1}}(\Fgrad\phi) V^{-\frac{q-a-1}{s-q+1}}\ud m\ud t,\\
		I_{q,2}&=a^{-\frac{(q-1)(s-a)}{s-q+1}}\int_{E_{2^{\frac{1}{\theta_2}}jR}\minus E_{jR}} \phi^{\frac{q(s-a)}{s-q+1}} F^{\frac{q(s-a)}{s-q+1}}(\Fgrad\eta_j) V^{-\frac{q-a-1}{s-q+1}}\ud m\ud t,\\
		I_{t,1}&=\int_{S\minus E_R}\norm{\pt_t\phi}^{\frac{s-a}{s-1}} V^{-\frac{1-a}{s-1}}\ud m\ud t,
	\end{aligned} \]
and
\begin{align*}
I_{t,2}&=\int_{E_{2^{\frac{1}{\theta_2}}jR}\minus E_{jR}} \phi^{\frac{s-a}{s-1}}\norm{\pt_t\eta_j}^{\frac{s-a}{s-1}} V^{-\frac{1-a}{s-1}}\ud m\ud t.
\end{align*}
	
	We now estimate these integrals, beginning with $I_{p,1}$. From the definition of $\phi$ and the fact that $R^{a}=e$, we have, for almost every $(x,t)\in S$,
	\[ F(\Fgrad\phi) \leq C\finv Aa\theta_2\sibb{\frac{r(x)^{\theta_2} +t^{\theta_1}}{R^{\theta_2}}}^{-Aa-1}\frac{r(x)^{\theta_2-1}}{R^{\theta_2}}
	\leq Ca\sibb{r(x)^{\theta_2}+ t^{\theta_1}}^{-A a-1}r(x)^{\theta_2-1}, \]
	which implies
	\[ \begin{aligned}
		I_{p,1} &\leq Ca^{-\frac{(p-1)(s-a)}{s-p+1}}\int_{S \minus E_R}\sibb{ a(r(x)^{\theta_2}+ t^{\theta_1})^{-A a-1}r(x)^{\theta_2-1} }^{\frac{p(s-a)}{s-p+1}}V^{-\frac{p-a-1}{s-p+1}} \ud m\ud t\\
		&\leq Ca^{\frac{p(s-a)-(p-1)(s-a)}{s-p+1}}\int_{S \minus E_R}\sibb{r(x)^{\theta_2}+ t^{\theta_1}}^{-(A a+1)\frac{p(s-a)}{s-p+1}} r(x)^{(\theta_2-1)\frac{p(s-a)}{s-p+1}}V^{-\frac{p-a-1}{s-p+1}} \ud m \ud t.
	\end{aligned} \]
	Applying Lemma~\ref{lem-inequality-decom-HP4} (2) with $\epsilon=a/(s-p+1)$, we obtain
	\[ I_{p,1} \leq Ca^{\frac{s-a}{s-p+1}}\int_{R/2^{\frac{1}{\theta_2}}}^{\infty}r^{-(A a+1)\theta_2\frac{p(s-a)}{s-p+1} + \bar{l}_{p} + C_0 \frac{a}{s-p+1} - 1}\sibb{\ln r}^{k_{p}} \ud r. \]
	
	Introduce the variable $y=\alpha\ln r$ with
	\[ \alpha=(Aa+1)\theta_2 \frac{p(s-a)}{s-p+1}-\bar{l}_{p}-C_0\frac{a}{s-p+1}. \]
	By the choice of $A>\sibb{\theta_2p+C_0 +1}/\sibb{\theta_2 p(s-a)}$, we have
	\[ -\alpha= -(Aa+1)\theta_2\frac{p(s-a)}{s-p+1} +\frac{sp\theta_2}{s-p+1} +C_0 \frac{a}{s-p+1} \leq-\frac{a}{s-p+1}.  \]
	Thus, we get
	\[ I_{p,1}\leq Ca^{\frac{s-a}{s-p+1}}\int_{0}^{\infty}\frac{1}{\alpha}\ue^{-y}\sibb{\frac{y}{\alpha}}^{k_p}\ud y \leq Ca^{\frac{s-a}{s-p+1}-k_p-1}. \]
	
	Turning to $I_{p,2}$, we observe that on $E_{2^{1/\theta_2}jR}\minus E_{jR}$,
	\[ \phi\leq j^{-Aa\theta_2 }\text{ and } F(\Fgrad\eta_j)\leq\frac{\finv\theta_2 r(x)^{\theta_2-1}}{(jR)^{\theta_2}}. \]
	Using the second inequality in HP4 with $\epsilon=a/\sibb{s-p+1}$, we get
	\[ \begin{aligned}
		I_{p,2}&\leq Ca^{-\frac{(p-1)(s-a)}{s-p+1}}\int_{E_{2^{\frac{1}{\theta_2}}jR}\minus E_{jR}}j^{-Aa\theta_2 \frac{p(s-a)}{s-p+1}}\sibb{\frac{r(x)^{\theta_2-1}}{(jR)^{\theta_2}}}^{\frac{p(s-a)}{s-p+1}}V^{-\frac{p-a-1}{s-p+1}}\ud m\ud t\\
		&= Ca^{-\frac{(p-1)(s-a)}{s-p+1}}j^{-Aa\theta_2 \frac{p(s-a)}{s-p+1}}(jR)^{-\theta_2\frac{p(s-a)}{s-p+1}}\int_{E_{2^{\frac{1}{\theta_2}}jR}\minus E_{jR}}r(x)^{(\theta_2-1)\frac{p(s-a)}{s-p+1}}V^{-\frac{p-a-1}{s-p+1}}\ud m\ud t \\
		&\leq Ca^{-\frac{(p-1)(s-a)}{s-p+1}}j^{-Aa\theta_2 \frac{p(s-a)}{s-p+1}}(jR)^{-\theta_2\frac{p(s-a)}{s-p+1}}(jR)^{\bar l_p+C_0\frac{a}{s-p+1}}\sibb{\ln(jR)}^{k_p}.
	\end{aligned} \]
	Since $\bar{l}_p=sp\theta_2/\sibb{s-p+1}$ and $R^a=\ue$, the powers of $R$ cancel out, and we obtain
	\[ I_{p,2}\leq Ca^{-\frac{(p-1)(s-a)}{s-p+1}}j^{-Aa\theta_2 \frac{p(s-a)}{s-p+1}+\theta_2\frac{ap}{s-p+1}+C_0\frac{a}{s-p+1}}\sibb{\ln(jR)}^{k_p}. \]
	Again, by the choice of $A>(\theta_2 p+C_0 +1)/\sibb{\theta_2p(s-a)}$ and for $a$ small enough,
	\[ -Aa \theta_2 \frac{p(s-a)}{s-p+1}+\theta_2\frac{ap}{s-p+1}+C_0\frac{a}{s-p+1}\leq -\frac{a}{s-p+1}, \]
	and therefore
	\[ I_{p,2}\leq Ca^{-\frac{(p-1)(s-a)}{s-p+1}}j^{-\frac{a}{s-p+1}}\sibb{\ln(jR)}^{k_p}. \]
	
	The quantities $I_{q,1}$ and $I_{q,2}$ are handled in exactly the same way, replacing $p$ by $q$ and $k_p$ by $k_q$. So we obtain
	\[ I_{q,1}\leq Ca^{\frac{s-a}{s-q+1}-k_q-1},\qquad I_{q,2}\leq Ca^{-\frac{(q-1)(s-a)}{s-q+1}} j^{-\frac{a}{s-q+1}}\sibb{\ln(jR)}^{k_q}. \]
	
	Next we estimate $I_{t,1}$ and $I_{t,2}$. From the definition of $\phi$, we get
	\[ \norm{\pt_t\phi}\leq Ca\sibb{r(x)^{\theta_2}+t^{\theta_1}}^{-Aa-1}t^{\theta_1-1}. \]
	Thus, by Lemma~\ref{lem-inequality-decom-HP4} (1) with $\epsilon=a/(s-1)$, we have
	\[ \begin{aligned}
		I_{t,1}&=\int_{S\minus E_R}\norm{\pt_t\phi}^{\frac{s-a}{s-1}} V^{-\frac{1-a}{s-1}}\ud m\ud t\\
		&\leq C\int_{S\minus E_R} \sibb{a\sibb{r(x)^{\theta_2}+t^{\theta_1}}^{-Aa-1}t^{\theta_1-1}}^{\frac{s-a}{s-1}} V^{-\frac{1-a}{s-1}}\ud m\ud t \\
		&\leq Ca^{\frac{s-a}{s-1}}\int_{S\minus E_R} \sibb{r(x)^{\theta_2}+t^{\theta_1}}^{-(Aa+1)\frac{s-a}{s-1}} t^{(\theta_1-1)\frac{s-a}{s-1}} V^{-\frac{1-a}{s-1}}\ud m\ud t \\
		&\leq Ca^{\frac{s-a}{s-1}}\int_{R/2^{\frac{1}{\theta_2}}}^{\infty}
		r^{-\theta_2(Aa+1)\frac{s-a}{s-1}+\bar s_1+C_0\frac{a}{s-1}-1}
		\sibb{\ln r}^{s_2} \ud r.
	\end{aligned} \]
	Let $y=\alpha\ln r$ with
	\[ \alpha=\theta_2(Aa+1)\frac{s-a}{s-1}-\bar s_1-C_0\frac{a}{s-1}. \]
	For $A>(\theta_2 +C_0 +1)/\sibb{\theta_2(s-a)}$ and  small $a$, we see
	\[ -\alpha\leq -\frac{a}{s-1}. \]
	Hence, we have
	\[ I_{t,1} \leq C a^\frac{s-a}{s-1}\int_{R/2^{\frac{1}{\theta_2}}}^{\infty}\ue^{-y}\sibb{ \frac{y}{\alpha} }^{s_2} \frac{1}{\alpha} \ud y \leq C a^{\frac{s-a}{s-1}-s_2 -1}. \]
	
	For $I_{t,2}$, on the annulus $E_{2^{1/\theta_2}jR}\minus E_{jR}$, we have
	\[ \norm{\pt_t\eta_j}\leq\frac{\theta_1 t^{\theta_1-1}}{(jR)^{\theta_2}}. \]
	Then using the first inequality in HP4 with $\epsilon=a/\sibb{s-1}$, we get
	\[ \begin{aligned}
		I_{t,2}&\leq \int_{E_{2^{\frac{1}{\theta_2}}jR}\minus E_{jR}}(j^{-\theta_2 A a})^{\frac{s-a}{s-1}}
		\sibb{\frac{\theta_1 t^{\theta_1-1}}{(jR)^{\theta_2}}}^{\frac{s-a}{s-1}}V^{-\frac{1-a}{s-1}}\ud m\ud t\\
		&\leq C (jR)^{-\theta_2\frac{s-a}{s-1}} j^{-A\theta_2 a\frac{s-a}{s-1}}(jR)^{\bar s_1+C_0\frac{a}{s-1}}\sibb{\ln(jR)}^{s_2}\\
		&\leq C j^{-\theta_2\frac{s-a}{s-1}-Aa\theta_2\frac{s-a}{s-1}+\frac{s\theta_2}{s-1}+\frac{C_0 a}{s-1}}\sibb{\ln(jR)}^{s_2}.
	\end{aligned} \]
	Since $A>(\theta_2 + C_0 +1)/\sibb{\theta_2(s-a)}$, the exponent of $j$ is bounded above by $-a/\sibb{s-1}$, which yields
	\[ I_{t,2}\leq C j^{-\frac{a}{s-1}}\sibb{\ln(jR)}^{s_2}. \]
	
	Now substituting all the estimates into \eqref{eq-inequality-integral-HP4-main-temp1}, we have
	\[ \begin{aligned}
		\int_{S}Vu^{-a+s}\phi_j^{b}\ud m\ud t
		\leq&\,C\sibb{a^{\frac{s-a}{s-p+1}-k_p-1}+a^{-\frac{(p-1)(s-a)}{s-p+1}} j^{-\frac{a}{s-p+1}}(\ln(jR))^{k_p}}\\
		&\quad+C\sibb{a^{\frac{s-a}{s-q+1}-k_q-1}+a^{-\frac{(q-1)(s-a)}{s-q+1}} j^{-\frac{a}{s-q+1}}(\ln(jR))^{k_q}}\\
		&\quad+C\sibb{a^{\frac{s-a}{s-1}-s_2-1}+j^{-\frac{a}{s-1}}\sibb{\ln(jR)}^{s_2}}.
	\end{aligned} \]
	
	Letting $j\to\infty$, the terms containing $j$ disappear and we obtain
	\[ \int_{E_R}Vu^{-a+s}\ud m\ud t\leq C\sibb{a^{\frac{s-a}{s-p+1}-k_p-1}+a^{\frac{s-a}{s-q+1}-k_q-1}+a^{\frac{1-a}{s-1}-s_2}}. \]
	
	Because $s>p-1$, $a^a\leq C$ and \[ \frac{s-a}{s-p+1}-\frac{p-1}{s-p+1}=1-\frac{a}{s-p+1},\]
	we have
	\[ \int_{E_R}V u^{-a+s}\ud m\ud t\leq C\sibb{a^{\bar k_p-k_p}+a^{\bar k_q-k_q}+a^{\bar{s}_2-s_2}}. \]
	
	Finally, by the choice $k_p<\bar k_p$, $k_q<\bar k_q$, $s_2<\bar s_2$, using Fatou's lemma by letting $R\to\infty$, we conclude
	\[ \int_{S}V u^{s}\ud m\ud t=0, \]
	which implies $u\equiv0$ almost everywhere in $S$.

\subsection{The case of HP5}
	The case of HP5 is the parabolic counterpart of HP2. We first prove that $Vu^s$ is integrable over the whole space-time $S$ by using the parabolic H\"{o}lder-type estimate (Lemma~\ref{lem-integralIdentity-cutFunction-holder-parabolic}), and then show that the integral over shrinking exteriors forces the solution to be zero almost everywhere.
	
	Choose a fixed constant $b$ independent of $R$ such that
	\[ b>\max\set{\frac{s}{s-1}, \frac{2sp}{s-p+1}}, \]
	and let $a:=1/\ln R$ for every sufficiently large $R$. Hence, Lemma~\ref{lem-integralIdentity-cutFunction-holder-parabolic} can be applied with these $a, b$. Then the constant $C(b)$ is an absolute constant.
	
	Define the cut-off function $\phi_j = \phi \eta_j$ exactly as in the proof of the case of HP4.
	Choose the constant $A$ large enough so that
	\[ A >\max_{z\in \set{p,q}}\set{\frac{2\sibb{\theta_2 z+C_0 +1}}{\theta_2zs},\frac{2\sibb{\theta_2+C_0 +1}}{\theta_2s},\frac{C_0+1}{\theta_2(s-p+1)}}, \]
	where $C_{0}$ is the constant appearing in HP5. Since $a<1/2$ and $s>1$, we have
	\[ A >\max_{z\in \set{p,q}}\set{\frac{\theta_2 z+C_0 +1}{\theta_2z(s-a)},\frac{\theta_2+C_0 +1}{\theta_2(s-a)},\frac{C_0+1}{\theta_2(s-p+1)}}. \]
	
	With $R$ large enough, the parameters $a,b$ satisfy the hypotheses of Lemma~\ref{lem-integralIdentity-cutFunction-holder-parabolic}.
	Applying that lemma with $\phi_j$ and noticing that $K=\set{(x,t):\phi_j(x,t)=1}=E_R$ for every $j$, we obtain
	\begin{equation}\label{eq-inequality-integral-HP5-main-temp1}
		\begin{aligned}
			\int_{S} V u^{s}\phi_j^{b}\ud m\ud t &\leq C(b)(a^{-1}Q_{j})^{\frac{p-1}{p}}\sibb{\int_{S\minus K} V u^{s}\phi_j^{b}\ud m\ud t}^{\frac{(a+1)(p-1)}{sp}} J_p^{\frac{s-(a+1)(p-1)}{sp}} \\
			&\quad + C(b)(a^{-1}Q_{j})^{\frac{q-1}{q}}\sibb{\int_{S\minus K} V u^{s}\phi_j^{b}\ud m\ud t}^{\frac{(a+1)(q-1)}{sq}} J_q^{\frac{s-(a+1)(q-1)}{sq}}\\
			&\quad + C(b)\sibb{\int_{S\minus K} V u^{s}\phi_j^{b}\ud m\ud t}^{\frac{1}{s}} J_t^{\frac{s-1}{s}},
		\end{aligned}
	\end{equation}
	where
	\[ \begin{aligned}
		Q_{j}&:=a^{-\frac{(p-1)(s-a)}{s-p+1}}\int_{S} V^{-\frac{p-a-1}{s-p+1}} F^{\frac{p(s-a)}{s-p+1}}(\Fgrad\phi_j) \ud m \ud t + a^{-\frac{(q-1)(s-a)}{s-q+1}}\int_{S} V^{-\frac{q-a-1}{s-q+1}} F^{\frac{q(s-a)}{s-q+1}}(\Fgrad\phi_j) \ud m \ud t\\
		&\quad + \int_{S} V^{-\frac{1-a}{s-1}} |\pt_t \phi_j|^{\frac{s-a}{s-1}} \ud m \ud t,
	\end{aligned} \]
	\[ \begin{aligned}
		J_z &:= \int_{S\minus K} V^{-\frac{(a+1)(z-1)}{s-(a+1)(z-1)}} F^{\frac{sz}{s-(a+1)(z-1)}}(\Fgrad\phi_j) \ud m \ud t,\quad z\in\set{p,q},
	\end{aligned} \]
and
\begin{align*}
J_t := \int_{S\minus K} V^{-\frac{1}{s-1}} |\pt_t \phi_j|^{\frac{s}{s-1}} \ud m \ud t.
\end{align*}
	
	We first estimate $Q_{j}$.  Arguing exactly as for the estimates of $I_{p,1},I_{p,2},I_{t,1}$ and $I_{t,2}$ in the proof of the case of HP4, but using Lemma~\ref{lem-inequality-decom-HP5} instead of Lemma~\ref{lem-inequality-decom-HP4}, we obtain
\begin{align*}
 Q_{j} \leq& C + C a^{-\frac{(p-1)(s-a)}{s-p+1}}j^{-\frac{a}{s-p+1}}(\ln(jR))^{\bar k_{p}}\\
&	+ C a^{-\frac{(q-1)(s-a)}{s-q+1}}j^{-\frac{a}{s-q+1}}(\ln(jR))^{\bar k_{q}}
	+ C j^{-\frac{a}{s-1}}(\ln(jR))^{\bar s_{2}}.
\end{align*}
	Letting $j\to\infty$, the three terms containing $j$ disappear and we conclude
	\[ \limsup_{j\to\infty} Q_{j} \leq C. \]
	
	Next we turn to the estimates of $J_{p}$, $J_{q}$ and $J_{t}$.  We treat $J_{p}$ in detail; the treatment of the other two terms is completely analogous.
	Using the pointwise inequality
\begin{align*}
F^{k}(\Fgrad\phi_j)\leq C\sibb{F^{k}(\Fgrad\phi)+\phi^{k}F^{k}(\Fgrad\eta_j)}
\end{align*}
with $k>0$, we split
	\[ J_{p} \leq C\sibb{J_{p,1}+J_{p,2}}, \]
	where
	\[ J_{p,1}:=\int_{S\minus E_R}V^{-\frac{(a+1)(p-1)}{s-(a+1)(p-1)}}F^{\frac{sp}{s-(a+1)(p-1)}}(\Fgrad\phi)\ud m\ud t, \]
and
	\[ J_{p,2}:=\int_{E_{2^{\frac{1}{\theta_{2}}}jR}\minus E_{jR}}V^{-\frac{(a+1)(p-1)}{s-(a+1)(p-1)}}\phi^{\frac{sp}{s-(a+1)(p-1)}}F^{\frac{sp}{s-(a+1)(p-1)}}(\Fgrad\eta_j)\ud m\ud t. \]
	
	Set
	\[ \epsilon_{p} := \frac{as(p-1)}{(s-p+1)(s-(p-1)(a+1))}, \]
	so that $-(a+1)(p-1)/\sibb{s-(a+1)(p-1)}=-\bar{k}_{p}-\epsilon_{p}$.
	
	A direct computation shows that
	\[ \frac{sp}{s-(a+1)(p-1)}=p\epsilon_p+s_{p} \]
	with $s_{p}=sp/\sibb{s-p+1}$. Substituting the gradient bound
	\[ F\sibb{\Fgrad\phi(x,t)}\leq C a\sibb{r^{\theta_2}+t^{\theta_1}}^{-Aa-1}r^{\theta_2-1} \quad\text{a.e. in }S\minus E_{R} \]
	and the definition of $\epsilon_p$ into $J_{p,1}$, we obtain
	\[ \begin{aligned}
		J_{p,1} &\leq C a^{\frac{sp}{s-(a+1)(p-1)}} \int_{S\minus E_{R}} V^{-\bar k_{p}-\epsilon_p} \sibb{r^{\theta_2}+t^{\theta_1}}^{-(Aa+1)\frac{sp}{s-(a+1)(p-1)}} r^{(\theta_2-1)\frac{sp}{s-(a+1)(p-1)}}\ud m\ud t \\
		&= C a^{p\epsilon_p+s_{p}} \int_{S\minus E_{R}} V^{-\bar k_{p}-\epsilon_p} \sibb{r^{\theta_2}+t^{\theta_1}}^{-(Aa+1)(p\epsilon_p+s_{p})} r^{(\theta_2-1)(p\epsilon_p+s_{p})}\ud m\ud t .
	\end{aligned} \]
	
	Applying the second inequality in Lemma~\ref{lem-inequality-decom-HP5} (2) with $\epsilon=\epsilon_{p}$ and
	$f(r)=r^{-\theta_2(Aa+1)(p\epsilon_p+s_{p})}$, we obtain
	\[ \begin{aligned}
		J_{p,1} &\leq C a^{p\epsilon_p+s_{p}} \iint{R/2^{1/\theta_2}}{\infty}{r^{-\theta_2(Aa+1)(p\epsilon_p+s_{p})+\bar l_{p}+C_{0}\epsilon_p-1}\sibb{\ln r}^{\bar k_{p}}}{r}.
	\end{aligned} \]
	
	Define
	\[ \alpha:=\theta_2(Aa+1)\sibb{p\epsilon_p+s_{p}}-\bar{l}_{p}-C_{0}\epsilon_p. \]
	By the choice of $A>\sibb{C_0 +1}/\sibb{\theta_2(s-p+1)}$, we have
	\[ A\theta_2 \geq \frac{C_0(p-1)}{p(s-p+1)}+\frac{(s-(a+1)(p-1))}{(s-p+1)^2}.\]
	Hence, we get \[ \alpha \geq \frac{asp}{(s-p+1)^2}. \]
	The change of variables $y=\alpha\ln r$ gives
	\[ \begin{aligned}
		J_{p,1} &\leq C a^{p\epsilon_p+s_{p}} \int_{\alpha\ln(R/2^{1/\theta_2})}^{\infty} \ue^{-y}\sibb{\frac{y}{\alpha}}^{\bar{k}_{p}}
		\frac{\ud y}{\alpha}\leq C a^{p\epsilon_p+s_{p}} \alpha^{-1-\bar{k}_{p}} \int_{0}^{\infty}\ue^{-y}y^{\bar k_{p}}\ud y \\
		&\leq C a^{p\epsilon_p+s_{p}-\bar{k}_{p}-1}
		\leq C a^{\frac{sp}{s-(a+1)(p-1)}-\bar{k}_{p}-1}.
	\end{aligned} \]
	In the last step we used $\alpha\geq c a$ and the relation
\begin{align*}
p\epsilon_p+s_{p}=\frac{sp}{s-(a+1)(p-1)}.
\end{align*}
This completes the estimate of $J_{p,1}$.
	
	For $J_{p,2}$ we work on the annulus $E_{2^{1/\theta_2}jR}\minus E_{jR}$.
	Using $\phi\leq j^{-Aa\theta_2}$ and $F(\Fgrad\eta_j)\leq C r^{\theta_2-1}/(jR)^{\theta_2}$, we obtain
	\[ \begin{aligned}
		J_{p,2} \leq C j^{-Aa\theta_2(p\epsilon_p+s_{p})} (jR)^{-\theta_2(p\epsilon_p+s_{p})} \int_{E_{2^{1/\theta_2}jR}\minus E_{jR}} r^{(\theta_2-1)(p\epsilon_p+s_{p})} V^{-\bar k_p-\epsilon_p} \ud m\ud t .
	\end{aligned} \]
	For the integral in the right-hand side, we apply the third inequality in HP5 with $\epsilon=\epsilon_p$. This yields
	\[ \int_{E_{2^{1/\theta_2}jR}\minus E_{jR}} r^{(\theta_2-1)(p\epsilon_p+s_p)} V^{-\bar k_p-\epsilon_p} \ud m\ud t \leq C (jR)^{\bar l_p+C_0\epsilon_p}(\ln(jR))^{\bar{k}_p}. \]
	Since $\bar{l}_p=s_p\theta_2$, the powers of $jR$ combine as
	\[ -\theta_2\sibb{p\epsilon_p+s_{p}}+\bar{l}_p+C_0\epsilon_p= (C_0-p\theta_2 )\epsilon_{p}. \]
	Therefore, using $R^a=\ue$, we have
	\[ J_{p,2}\leq C j^{- Aa\theta_2(p\epsilon_p +s_p)+ (C_{0}-p\theta_{2})\epsilon_{p}}
	(\ln(jR))^{\bar k_p}. \]
	By our choice of $A>\sibb{C_0 +1}/\sibb{\theta_2(s-p+1)}$, we get
\[ A\theta_2 \geq \frac{s-(a+1)(p-1)}{(s-p+1)^2}+\frac{(p-1)(C_0-p\theta_2)}{p(s-p+1)}.\]
	As a result, the exponent of $j$ is strictly negative and larger in the absolute value than
	$asp/(s-p+1)^2$ for small $a$. Hence
	\[ J_{p,2}\leq C j^{-\frac{asp}{(s-p+1)^2}}(\ln(jR))^{\bar{k}_p}. \]
	
	Thus,
	\[ J_{p} \leq C a^{\frac{sp}{s-(a+1)(p-1)}-\bar{k}_{p}-1} + C j^{-\frac{asp}{(s-p+1)^2}}(\ln(jR))^{\bar{k}_{p}}, \]
	and consequently
	\[ \limsup_{j\to\infty} J_{p} \leq C a^{\frac{sp}{s-(a+1)(p-1)}-\bar k_{p}-1}. \]
	Exactly the same reasoning yields
	\[ J_{q} \leq C a^{\frac{sq}{s-(a+1)(q-1)}-\bar{k}_{q}-1} + C j^{-\frac{asq}{(s-q+1)^2}}\sibb{\ln(jR)}^{\bar{k}_{q}}, \]
	and
	\[ \limsup_{j\to\infty} J_{q} \leq C a^{\frac{sq}{s-(a+1)(q-1)}-\bar k_{q}-1}. \]
	
	The estimate of $J_{t}$ is similar to that of $J_{p}$. First, we have
	\[ J_{t}\leq C \sibb{J_{t,1}+J_{t,2}}, \]
	where
	\[ J_{t,1}:=\int_{S\minus E_R} V^{-\frac{1}{s-1}}|\pt_t\phi|^{\frac{s}{s-1}}\ud m\ud t, \]
and
\begin{align*}
	J_{t,2}:=\int_{E_{2^{1/\theta_2}jR}\minus E_{jR}} V^{-\frac{1}{s-1}}\phi^{\frac {s}{s-1}}|\pt_t\eta_j|^{\frac{s}{s-1}}\ud m\ud t.
\end{align*}
	Using the pointwise bounds
\begin{align*}
|\pt_t\phi|&\leq Ca (r^{\theta_2}+t^{\theta_1})^{-Aa-1}t^{\theta_1-1},\\
\phi&\leq j^{-A\theta_2a},\\
|\pt_t\eta_j|&\leq C t^{\theta_1-1}/(jR)^{\theta_2},
\end{align*}
	and applying Lemma~\ref{lem-inequality-decom-HP5} (1) with $\epsilon=0$ (justified by Fatou's lemma through $\epsilon\to0$) as in the HP4 case,
	we obtain
	\[ J_{t,1}\leq C,\quad J_{t,2}\leq C j^{-\frac{a}{s-1}}(\ln(jR))^{\bar s_2}. \]
	Consequently $\limsup_{j\to\infty} J_{t}\leq C$.
	
	Finally, substituting these estimates into \eqref{eq-inequality-integral-HP5-main-temp1}, we have
	\begin{align}\label{eq-inequality-integral-HP5-main-temp2}
		\int_{S}Vu^{s}\phi_j^b\ud m\ud t
		&\leq C\sibb{\int_{S\minus E_R}Vu^{s}\phi_j^b\ud m\ud t}^{\frac1s}+Ca^{\frac{a(p-1)}{p(s-p+1)}}\sibb{\int_{S\minus E_R}Vu^{s}\phi_j^b\ud m\ud t}^{\frac{(a+1)(p-1)}{sp}}\nonumber\\
		&\quad +Ca^{\frac{a(q-1)}{q(s-q+1)}}\sibb{\int_{S\minus E_R}Vu^{s}\phi_j^b\ud m\ud t}^{\frac{(a+1)(q-1)}{sq}}.
	\end{align}
As in the discussion of Section~\ref{sec-the case of HP2}, we then obtain
	\begin{align*}
		1+\int_{S}Vu^{s}\phi_j^b\ud m\ud t
		&\leq C\sibb{1+\int_{S\minus E_R}Vu^{s}\phi_j^b\ud m\ud t}^{\frac1s}+Ca^{\frac{a(p-1)}{p(s-p+1)}}\sibb{1+\int_{S\minus E_R}Vu^{s}\phi_j^b\ud m\ud t}^{\frac{(a+1)(p-1)}{sp}}\nonumber\\
		&\quad +Ca^{\frac{a(q-1)}{q(s-q+1)}}\sibb{1+\int_{S\minus E_R}Vu^{s}\phi_j^b\ud m\ud t}^{\frac{(a+1)(q-1)}{sq}}.
	\end{align*}
	
	Set
	\[ \gamma=\max\set{\frac{(a+1)(p-1)}{sp},\frac{(a+1)(q-1)}{sq},\frac{1}{s}}. \]
	Since $a^{a(p-1)/(p(s-p+1))}<C$, $a^{a(q-1)/(q(s-q+1))}<C$ and $\gamma< 1$ for sufficiently small $a$, the local integrability of $V u^{s}\in L_{\mathrm{loc}}^{1}(S)$ implies
	\[ \begin{aligned}
		\sibb{1+\int_{S}{V u^{s}\phi_j^{b}}\ud m\ud t}^{1-\gamma}&\leq C.
	\end{aligned} \]
	Letting $j\to\infty$, we deduce, for any sufficiently large $R$,
	\[ \int_{E_R}V u^{s}\ud m\ud t\leq C. \]
	
	 Furthermore, letting $j\to\infty$ in \eqref{eq-inequality-integral-HP5-main-temp2}, we see
	\begin{equation}\label{eq-inequality-integral-HP5-temp3}
		\begin{aligned}
			\int_{E_R}V u^{s}\ud m\ud t
			&\leq C\sibb{\int_{S\minus E_R}Vu^{s}\ud m\ud t}^{\frac1s}+C\sibb{\int_{S\minus E_R}{V u^{s}}\ud m\ud t}^{\frac{(a+1)(p-1)}{sp}}\\
			&\quad+C \sibb{\int_{S\minus E_R}{V u^{s}}\ud m\ud t}^{\frac{(a+1)(q-1)}{sq}}.
		\end{aligned}
	\end{equation}
	
	Finally, let $R\to\infty$ in \eqref{eq-inequality-integral-HP5-temp3}. The integral over $S\minus E_R$ tends to zero because the whole integral $\int_{S} V u^{s}dm dt$ is finite. The right-hand side therefore tends to zero, while the left-hand side increases to $\int_{S} V u^{s}dm dt$.  Hence
	\[ \int_{S}Vu^{s}\ud m\ud t = 0, \]
	which forces $u=0$ a.e. in $S$.
	
\section{Final Remarks}
	We have proved that on a forward geodesically complete noncompact Finsler measure space with finite reversibility, any nonnegative weak solution of the $(p,q)$-Laplacian elliptic inequality (resp. parabolic inequality) vanishes almost everywhere, provided the potential satisfies one of the integral growth conditions HP1--HP3 (resp. HP4--HP5). Next we make two remarks here.
	
	\begin{rem}
		The weighted Riemannian manifolds are included in our framework as a special case. Indeed, let $(M,g)$ be a complete noncompact Riemannian manifold and let
		\[ F(x,X)=\sqrt{g_x(X,X)},\ (x,X)\in TM, \qquad \ud m=\ue^{-f}\ud\mathrm{vol}_g,\]
		where $f\in C^\infty(M)$ and $\ud\mathrm{vol}_g$ is the volume form induced by $g$.
		Then $(M,F,m)$ is a reversible Finsler measure space. It is not hard to check that the Finsler gradient of a function agrees with the Riemannian gradient.
		Thus, the corresponding weighted elliptic and parabolic Liouville theorems can be obtained on weighted Riemannian manifolds.
	\end{rem}
	
	On the other hand, we note that the proofs in this paper work essentially unchanged for inequalities that contain a finite sum of Finsler $p_i$-Laplacians with the corresponding Finsler structures. More precisely, we have the following remark.
	\begin{rem}
		Here, we only present the elliptic case in some details. The parabolic case is completely analogous.
		
		Let $F_1,\dots,F_k$ be smooth Finsler structures on the given Finsler measure space $(M,F,m)$ and suppose there exist constants $C_i>0$, $1\leq i\leq k$, such that
		\begin{equation}\label{eq:dual-compare}
			F_i^{*}(x,\xi)\leq C_i\,F^{*}(x,\xi)\text{ for all }(x,\xi)\in T^{*}M,\ i=1,\dots,k.
		\end{equation}
		Consider the elliptic inequality
		\begin{equation}\label{eq:multi-elliptic}
			\sum_{i=1}^{k}\laplace^{m,F_i}_{p_i}u+V(x)u^{s}\leq 0,\quad p_i>1,\quad s>\max_{1\leq i\leq k}\{p_i\}-1,
		\end{equation}
		where $V\in L^{1}_{\mathrm{loc}}(M)$ is strictly positive almost everywhere.
		
		A nonnegative function $u\in W^{1,\max_{1\leq i\leq k} \{p_i\}}_{\mathrm{loc}}(M)$ is a weak solution of \eqref{eq:multi-elliptic} if
		\begin{equation*}
			\qint{M}{Vu^{s}\psi}{m} \leq\sum_{i=1}^{k}\qint{M}{F_i^{p_i-2}(\grad_{F_i}u)\ud\psi(\grad_{F_i}u)}{m}
		\end{equation*}
		for every nonnegative test function $\psi\in W^{1,\max_{1\leq i\leq k} \{p_i\}}(M)\intersect L^\infty(M)$ with compact support.
		
		Since $\ud\psi(\grad_{F_i}u)\leq F_i(\grad_{F_i}\psi)F_i(\grad_{F_i}u)$, the same arguments as before yield Lemmas~\ref{lem-integralIdentity-cutFunction-young} and~\ref{lem-integralIdentity-cutFunction-holder}.
		More precisely, under the conditions of $a$ and $b$ stated as in those lemmas, we have the Caccioppoli-type inequality
		\begin{equation}\label{eq:Caccioppoli-multi}
			\begin{aligned}
				&\sum_{i=1}^{k}a\qint{M}{u^{-1-a}\phi^{b}F_i^{p_i}(\grad_{F_i}u)}{m} +\qint{M}{Vu^{s-a}\phi^{b}}{m}\\
				&\leq\sum_{i=1}^{k} C(b) a^{-\frac{(p_i-1)(s-a)}{s-p_i+1}} \qint{M}{V^{-\frac{p_i-a-1}{s-p_i+1}}F_i^{\frac{p_i(s-a)}{s-p_i+1}}(\grad_{F_i}\phi)}{m},
			\end{aligned}
		\end{equation}
		and the H\"{o}lder-type inequality
		\begin{equation}\label{eq:Holder-multi}
			\begin{aligned}
				\qint{M}{Vu^{s}\phi^{b}}{m}&\leq\sum_{i=1}^{k} C(b)(a^{-1}Q)^{\frac{p_i-1}{p_i}}\sibb{\qint{M\minus K}{Vu^{s}\phi^{b}}{m}}^{\frac{(a+1)(p_i-1)}{sp_i}} J_{p_i}^{\frac{s-(a+1)(p_i-1)}{sp_i}},
			\end{aligned}
		\end{equation}
		where $K=\set{x:\phi(x)=1}$,
		\[ \begin{aligned}
			Q&=\sum_{i=1}^{k} a^{-\frac{(p_i-1)(s-a)}{s-p_i+1}}\qint{M}{V^{-\frac{p_i-a-1}{s-p_i+1}} F_i^{\frac{p_i(s-a)}{s-p_i+1}}(\grad_{F_i}\phi)}{m},\\
			J_{p_i} &= \qint{M\minus K}{V^{-\frac{(a+1)(p_i-1)}{s-(a+1)(p_i-1)}}F_i^{\frac{p_is}{s-(a+1)(p_i-1)}}(\grad_{F_i}\phi)}{m} .
		\end{aligned} \]
		
		For any $\phi\in C^{1}(M)$, the inequality~\eqref{eq:dual-compare} implies
		\[ F_i(\grad_{F_i}\phi)=F_i^{*}(\ud\phi)\le C_iF^{*}(\ud\phi)=C_iF(\grad_F\phi). \]
		Consequently, each integral containing $\grad_{F_{i}}\phi$ is bounded above by the corresponding integral involving $\grad_{F}\phi$. Hence, the right-hand sides of \eqref{eq:Caccioppoli-multi} and \eqref{eq:Holder-multi} can be bounded by the analogous terms of $F$ up to multiplicative constants.
		
		Lemma~\ref{lem-inequality-decom-elliptic} depends only on the Finsler measure space $(M,F,m)$ and HP1--HP3, and therefore it remains unchanged.
		
		Hence, the proof of Theorem~\ref{thm-main-elliptic} goes through without any essential change.
		First, we use the same cut-off functions built from $r(x):=d_F(x_0,x)$, and apply \eqref{eq:Caccioppoli-multi} or \eqref{eq:Holder-multi} to obtain the inequality with respect to $F_{i}$.
		Next, we replace the terms of $\grad_{F_{i}}$ by those of $\grad_{F}$ via \eqref{eq:dual-compare}, and estimate the integrals of $F$ exactly as in the $(p,q)$-Laplacian case. The only extra requirement is that $V$ satisfy one of the conditions HP1, HP2 and HP3 for every exponent $z=p_i$, $i=1,\dots,k$.
		
		The parabolic case can be handled similarly. Therefore, we can extend the Liouville theorems to the case of any finite sum of Finsler $p_i$-Laplacians.
	\end{rem}

\end{document}